\documentclass[11pt,reqno]{amsart}

\usepackage{amsmath}
\usepackage{amsthm}
\usepackage{amssymb}
\usepackage[all]{xy}
\usepackage{mathrsfs}
\usepackage[latin1]{inputenc}
\usepackage{pstricks}
\usepackage{a4wide}

\title[Asymptotically cylindrical $G_2$-manifolds]{Asymptotically
  cylindrical 7-manifolds of holonomy $G_2$ with applications to compact
  irreducible $G_2$-manifolds}
\author{Alexei Kovalev}
\author{Johannes Nordström}
\thanks{The second author is partially supported by the Royal Swedish Academy
  of Sciences funds.}

\address{DPMMS, University of Cambridge,
    Centre for \mbox{Mathematical} Sciences,
    \mbox{Wilberforce} Road, Cambridge CB3 0WB, UK}
\email{a.g.kovalev@dpmms.cam.ac.uk}
\address{Department of Mathematics, South Kensington Campus,
    Imperial College London, London SW7 2AZ, UK}
\email{j.nordstrom@imperial.ac.uk}

\DeclareMathOperator{\const}{const}
\DeclareMathOperator{\Ker}{Ker}
\DeclareMathOperator{\vol}{vol}

\newcommand{\half}{{\textstyle\frac{1}{2}}}
\newcommand{\third}{{\textstyle\frac{1}{3}}}
\newcommand{\quart}{{\textstyle\frac{1}{4}}}
\newcommand{\tquart}{{\textstyle\frac{3}{4}}}

\newcommand{\bbz}{\mathbb{Z}}
\newcommand{\bbr}{\mathbb{R}}
\newcommand{\bbc}{\mathbb{C}}
\newcommand{\bbrp}{\mathbb{R}_{+}}

\newcommand{\tpar}{s}
\newcommand{\defstr}{\calm}

\newcommand{\gtstr}{$G_{2}$\nobreakdash-\hspace{0pt}structure}
\newcommand{\gtmfd}{$G_{2}$\nobreakdash-\hspace{0pt}manifold}
\newcommand{\ltwoorth}{$L^{2}$\nobreakdash-\hspace{0pt}orthogonal}
\newcommand{\cystr}{Calabi--Yau structure}
\newcommand{\sob}[1]{L^{2}_{#1}}
\newcommand{\sobd}[1]{L^{2}_{#1,\delta}}
\newcommand{\sobf}[1]{L^{14}_{#1}}
\newcommand{\sobp}[1]{L^{p}_{#1}}

\newcommand{\harm}{\mathcal{H}}
\newcommand{\norm}[1]{\Vert #1 \Vert}
\newcommand{\lnorm}[1]{\Vert #1 \Vert_{L^{2}}}

\newcommand{\cald}{\mathcal{D}}

\newcommand{\calx}{\mathcal{X}}
\newcommand{\calm}{\mathcal{M}}
\newcommand{\calf}{\mathcal{F}}
\newcommand{\cale}{\mathcal{E}}

\newcommand{\gen}[1]{\langle#1\rangle}
\newcommand{\inner}[1]{{<}#1{>}}
\newcommand{\contrax}[1]{\textstyle\frac{\partial}{\partial x^{#1}}}
\newcommand{\contrat}{\textstyle\frac{\partial}{\partial t}}
\newcommand{\contral}{\textstyle\frac{\partial}{\partial L}}
\newcommand{\dddt}{\frac{d}{dt}}
\newcommand{\dsubx}{d^{}_{\! \scriptscriptstyle X}}
\newcommand{\dcsubx}{d^{*}_{\! \scriptscriptstyle X}}
\newcommand{\tv}{\tilde \varphi}
\newcommand{\tbtmatrix}[4]{\left( {\setlength\arraycolsep{2pt}
\begin{array}{cc} #2 & #4 \\ #1 & #3 \end{array}} \right)}
\newcommand{\gi}{\pm}

\newcommand{\ycal}[1]{\mathcal{#1}_{y}}
\newcommand{\gical}[1]{\mathcal{#1}_{\gi}}
\newcommand{\onecal}[1]{\mathcal{#1}_{+}}
\newcommand{\twocal}[1]{\mathcal{#1}_{-}}
 
\newcommand{\CP}{\mathbb{C}P}
\newcommand{\CX}{\mathbb{C}}
\newcommand{\RE}{\mathbb{R}}
\newcommand{\ZE}{\mathbb{Z}}
\newcommand{\Ga}{\Gamma}
\newcommand{\eps}{\varepsilon}
\renewcommand{\epsilon}{\varepsilon}
\newcommand{\phin}{\varphi^{\text{init}}}
\newcommand{\p}{\partial}
\newcommand{\oW}{\overline{W}}
\newcommand{\we}{\wedge}

\newtheorem{thm}{Theorem}[section]
\newtheorem{prop}[thm]{Proposition}
\newtheorem{lem}[thm]{Lemma}
\newtheorem{cor}[thm]{Corollary}
\theoremstyle{definition}
\newtheorem{defn}[thm]{Definition}
\theoremstyle{remark}
\newtheorem{rmk}[thm]{Remark}
\newtheorem{ex}[thm]{Example}

\renewcommand{\theenumi}{\textup{(\roman{enumi})}}
\newcommand{\alphenumi}{\textup{(\alph{enumi})}}
\renewcommand{\labelenumi}{\theenumi}

\begin{document}

\begin{abstract}
We construct examples of exponentially asymptotically cylindrical (EAC)
Riemannian 7-manifolds with holonomy group equal to $G_2$. To our knowledge,
these are the first such examples. We also obtain EAC coassociative calibrated
submanifolds. Finally, we apply our results to show that one of the
compact $G_2$-manifolds constructed by Joyce by desingularisation of a flat
orbifold $T^7/\Gamma$ can be deformed to give one of the compact \linebreak
$G_2$-manifolds obtainable as a generalized connected sum of two EAC
$SU(3)$-manifolds via the method of~\cite{kovalev03}.
\end{abstract}

\maketitle

\section{Introduction}

The Lie group $G_2$ occurs as the holonomy group of the Levi--Civita
connection on some Riemannian 7-dimensional manifolds. The possibility of
holonomy~$G_2$ was suggested in Berger's classification of the Riemannian
holonomy groups~\cite{berger55}, but finding examples of metrics with holonomy
exactly $G_{2}$ is an intricate task. The first local examples were
constructed by Bryant \cite{bryant87} using the theory of exterior
differential systems, and complete examples were constructed by Bryant and
Salamon \cite{bryant89} and by Gibbons, Page and Pope~\cite{GPP}.
The first compact examples were constructed by Joyce \cite{joyce96} by
resolving singularities of finite quotients of flat tori and the method was
further developed in~\cite{joyce00}.

Later the first author~\cite{kovalev03} gave a different method of producing
new compact examples of 7-manifolds with holonomy~$G_2$ by gluing pairs of
{\em asymptotically cylindrical} manifolds. More precisely, a Riemannian
manifold is exponentially asymptotically cylindrical (EAC) if outside a
compact subset it is diffeomorphic to $X \times \RE_{>0}$ for some compact
$X$, and the metric is asymptotic to a product metric at an exponential
rate. An important part of the method in~\cite{kovalev03} is the proof of a
version of the Calabi conjecture for manifolds with cylindrical ends producing
EAC Ricci-flat Kähler $3$-folds $W$ with holonomy $SU(3)$. The product EAC
metric on a 7-manifold $W\times S^1$ then also has holonomy $SU(3)$, a maximal
subgroup of $G_{2}$, and is induced by a torsion-free \gtstr{}. In fact,
$W\times S^1$ cannot have an EAC metric with holonomy equal to~$G_2$
by~\cite[Theorem~3.8]{jn1} because the fundamental group of $W\times S^1$ is
not finite.

The purpose of this paper is to construct examples of exponentially
asymptotically cylindrical manifolds whose holonomy is exactly $G_{2}$. To our
knowledge these are first such examples. Note that the metrics with holonomy
$G_2$ in \cite{bryant89} are asymptotically conical and not EAC.

It is by now a standard fact that a metric with holonomy $G_2$ on a 7-manifold
can be defined in terms of a `stable' differential $3$-form $\varphi$
equivalent to a {\em\gtstr}. More generally, any \gtstr~$\varphi$ determines a
metric since $G_2$ is a subgroup of~$SO(7)$. This metric will have holonomy in
$G_{2}$ if the \gtstr{} is \emph{torsion-free}. The latter
condition is equivalent to the defining $3$-form $\varphi$ being closed and
coclosed, a non-linear first order PDE. A 7-manifold endowed with a
torsion-free \gtstr\ is called a {\em\gtmfd}. Thus a \gtmfd\ is a Riemannian
manifold with holonomy contained in~$G_2$. For compact or EAC \gtstr s there
is a simple topological criterion to determine if the holonomy is exactly
of~$G_2$. See \S\ref{prelimsec} for the details.

Joyce finds examples of \gtstr s on compact manifolds that have small torsion
by resolving singularities of quotients of a torus $T^{7}$ equipped with a
flat \gtstr{} by suitable finite groups~$\Gamma$.
The proof in \cite[Ch.~11]{joyce00} of the existence result for torsion-free
\gtstr s on compact 7-manifolds is carefully written to use the
compactness assumption as little as possible. A large part of the proof can
therefore be used in the EAC setting too. The main additional difficulty in
this case is to show that the \gtstr{} constructed has the desired exponential
rate of decay to its cylindrical asymptotic model. This task is accomplished
by our first main result Theorem~\ref{eacperturbthm}. In \S\ref{exsec} we
apply this result and explain how, in one particular example, one can cut
$T^{7}/\Gamma$ into two pieces along a hypersurface, attach a semi-infinite
cylinder to each half, and resolve the singularities to form EAC \gtstr s
satisfying the hypotheses of Theorem~\ref{eacperturbthm}. In a similar way, we
obtain in~\S\ref{furthersec} more examples of EAC \gtmfd s which are
simply-connected with a single end, and therefore have holonomy exactly
$G_{2}$ by~\cite[Theorem~3.8]{jn1} (see \S\ref{as.cyl}). This includes examples
both where the holonomy of the cross-section is $SU(3)$ and where it reduces
to (a finite extension of) $SU(2)$ or is flat. We explain how to compute their
Betti numbers, and find some examples of asymptotically cylindrical
coassociative minimal submanifolds.

In~\S\ref{pullsec} we study a kind of inverse of the above construction.
Given a pair of EAC \gtmfd s with asymptotic cylindrical models matching via
an orientation-reversing isometry, one can truncate their cylindrical ends
after some large length $L$ and identify their boundaries to form a
generalized connected sum, a compact manifold with an approximately
cylindrical neck of length approximately~$2L$. This compact 7-manifold
inherits from the pair of EAC \gtmfd s a well-defined \gtstr\ and the gluing
theorem in~\cite[\S 5]{kovalev03} asserts that when $L$ is sufficiently large
this \gtstr{} can be perturbed to a torsion-free one.

Our method of constructing EAC \gtmfd s by resolving `half' of $T^{7}/\Gamma$
produces them in such matching pairs. The connected sum of the pair is
topologically the same as the compact \gtmfd{} $(M,\varphi)$ obtained by
resolving the initial orbifold $T^{7}/\Gamma$. We show in our second main
result Theorem~\ref{pullthm} that $\varphi$ can be continuously deformed to
the torsion-free \gtstr s obtained by gluing  a pair of EAC \gtmfd s as
in~\cite{kovalev03}. In other words, the \gtstr s produced by the
connected-sum method lie in the same connected component of the moduli space
of torsion-free \gtstr s as the ones originally constructed by
Joyce. Informally, the path connecting $\varphi$ to the connected-sum
\gtstr s is given by increasing the length of one of the $S^{1}$ factors in
$T^{7}$ before resolving $T^{7}/\Gamma$. In this sense, the EAC \gtmfd s are
obtained by `pulling apart' the compact \gtmfd~$(M,\varphi)$.

In \S\ref{connectsums} we consider one pulling-apart example in
detail and identify the two EAC manifolds as products of $S^1$ and a complex
3-fold. The latter complex 3-folds were studied in~\cite{kovalev-lee08}
obtained from K3 surfaces with non-symplectic involution and the gluing
produces a compact $G_2$-manifold according to the method
of~\cite{kovalev03}. Thus the compact 7-manifold $M$ admits a path
$g(t)$, $0<t<\infty$ of metrics with holonomy $G_2$ so that the limit as $t\to
0$ corresponds to an orbifold $T^7/\Ga$ and the limit as $t\to\infty$
corresponds to a disjoint union of EAC \gtmfd s of the form $W_j\times S^1$,
$j=1,2$, where each $W_j$ is an EAC Calabi--Yau complex 3-fold with
holonomy $SU(3)$. To the authors' knowledge, $g(t)$ is the first example of
$G_2$-metrics on a compact manifold exhibiting two geometrically different
types of deformations, related to different constructions (\cite{joyce00}
and~\cite{kovalev03}) of compact irreducible \gtmfd s. (Demonstrating that two
constructions produce distinct examples of \gtmfd s can often be accomplished
by checking that these have different Betti numbers, a rather easier task.)

For the examples in this paper, we mostly restrict attention to one
compact 7-manifold underlying the \gtmfd s constructed in~\cite[I~\S
2]{joyce96}.  However, our techniques can be extended with more or less
additional work to construct more examples of EAC \gtmfd s from other \gtmfd
s, including those obtained in~\cite{joyce00} by resolving more complicated
singularities. The authors hope to develop this in a future paper.

\section{Preliminaries}
\label{prelimsec}

\subsection{Torsion-free $G_2$-structures and the holonomy group $G_{2}$}
\label{su2sub}

The group $G_{2}$ can be defined as the automorphism group of the normed
algebra of octonions. Equivalently, $G_{2}$ is the stabiliser in
$GL(\bbr^{7})$ of
\begin{equation}
\label{g2formeq}
\varphi_{0} = dx^{123} + dx^{145} + dx^{167} + dx^{246}
 - dx^{257} - dx^{347} - dx^{356} \in \Lambda^{3}(\bbr^{7})^{*} ,
\end{equation}
where $dx^{ijk}=dx^i\we dx^j\we dx^k$ \cite[p.~539--541]{bryant87}.
A \emph{\gtstr{}} on a 7-manifold $M$ may therefore be induced by a
choice of a differential $3$-form $\varphi$ such that
$\iota_p^*(\varphi(p))=\varphi_0$, for each $p\in M$ for some linear isomorphism
$\iota_p:\bbr^7\to T_pM$ smoothly depending on~$p$. Every such 3-form on $M$
will be called \emph{stable}, following~\cite{hitchin-stable}, and we shall,
slightly inaccurately, say that $\varphi$ \emph{is} a $G_2$-structure. As
$G_2\subset SO(7)$, a \gtstr{} induces a Riemannian metric $g_{\varphi}$ and
an orientation on $M$, and thus also a Levi--Civita connection
$\nabla_{\varphi}$ and a Hodge star~$*_{\varphi}$.

The holonomy group of a connected Riemannian manifold $M$ is defined up to
isomorphism as the group of isometries of a tangent space at $p\in M$
generated by parallel transport, with respect to the Levi--Civita connection,
around closed curves based at~$p$. Parallel tensor fields on a manifold
correspond to invariants of its holonomy group and the holonomy of $g_\varphi$
on~$M$ will be contained in~$G_{2}$ if and only if $\nabla_{\varphi}\varphi =
0$. A \gtstr{} satisfying this latter condition is called
\emph{torsion-free} and by a result of Gray (see \cite[Lemma 11.5]{salamon})
this is equivalent to
$$
d\varphi = 0 \text{ and }d{*_{\varphi}}\varphi = 0.
$$
We call a $7$-dimensional manifold equipped with a torsion-free \gtstr{}
a \emph{\gtmfd}. We call a \gtmfd\ {\em irreducible} of the holonomy of the
induced metric is all of $G_2$ (i.e.\ not a proper subgroup). A compact
\gtmfd\ is irreducible if and only if its fundamental group is finite
\cite[Propn.~10.2.2]{joyce00}.

More generally, the only connected Lie subgroups of $G_2$ that can arise as
holonomy of the Riemannian metric on a $G_2$-manifold are $G_2$, $SU(3)$,
$SU(2)$ and $\{1\}$ \cite[Theorem~10.2.1]{joyce00}.

We call a \gtstr{} $\varphi_X$ on a product manifold $X^{6} \times \bbr$
\emph{cylindrical} if it is translation-invariant in the second factor and
defines a product metric $g_M=dt^2+g_X$, where $t$ denotes the coordinate
on~$\bbr$. Then $\frac{\p}{\p t}$ is a parallel vector field on
$X^{6} \times \bbr$. The stabiliser in $G_{2}$ of a vector in $\bbr^{7}$ is
$SU(3)$, so the Riemannian product of $X^{6}$ with $\bbr$ has holonomy
contained in $G_{2}$ if and only if the holonomy of~$X$ is contained
in~$SU(3)$. 
The latter condition means that $X$ is a complex 3-fold with a Ricci-flat
Kähler metric and admits a nowhere-vanishing holomorphic (3,0)-form,
i.e.\ $X$ is a \emph{Calabi--Yau 3-fold}. More explicitly, we can write
\begin{equation}\label{g2su3}
\varphi_X = \Omega + dt \wedge \omega,
\quad\text{where }
\omega={\textstyle\frac{\p}{\p t}}\lrcorner\varphi_X
\text{ and }
\Omega=\varphi_X|_{X\times\{\text{pt}\}}.
\end{equation}
Then $\omega$ is the Kähler form on~$X$ and $\Omega$ is the real part of a
holomorphic $(3,0)$-form on~$X$, whereas $g(\varphi_X)=dt^2+g_X$. It can be
shown that a pair $(\Omega, \omega)$ of closed differential forms obtained
from a torsion-free $G_2$-structure as in~\eqref{g2su3} determines a
Calabi--Yau structure on~$X$ (cf.~\cite[Lemma 6.8]{hitchin-sl} and
\cite[Propn.~11.1.2]{joyce00}).

If the cross-section is itself a Riemannian product $X=S^1 \times S^1\times D$
then $D$ is a Calabi--Yau complex surface with holonomy in $SU(2)\cong Sp(1)$,
with Kähler form $\kappa_I$ and holomorphic (2,0)-form $\kappa_J+i\kappa_K$. 
Alternatively, $D$ may be described as a \emph{hyper-Kähler 4-manifold}, so
$D$ has three integrable complex structures $I,J,K$ satisfying quaternionic
relations $IJ=-JI=K$ and a metric which is Kähler with respect to all three.
The $\kappa_I$,$\kappa_J$,$\kappa_K$ are the respective Kähler forms and this
triple of closed real 3-forms in fact determines the hyper-Kähler structure
(see \cite[p.~91]{hitchin87}).
Denote by $x^1,x^2$ the coordinates on the two $S^1$ factors of~$X$. Then
the cylindrical torsion-free $G_2$-structure on
$\RE\times S^1\times S^1\times D$ corresponding to a hyper-Kähler structure
on~$D$ is
\begin{equation}\label{g2su2}
\varphi_D=dx^1\we dx^2\we dt+dx^1\we\kappa_I+dx^2\we\kappa_J+dt\we\kappa_K.
\end{equation}
It induces a product metric $g(\varphi_D)= dt^2+(dx^1)^2+(dx^2)^2+g_D$
(cf.~\cite[Propn.~11.1.1]{joyce00}).

\subsection{Asymptotically cylindrical manifolds}
\label{as.cyl}

A non-compact manifold $M$ is said to have \emph{cylindrical ends}
if $M$ is written as a union of a compact manifold $M_{0}$ with boundary
$\p M_0$ and a half-cylinder $M_{\infty}=\bbrp\times X$, the two pieces
identified via the common boundary
$\p M_{0}\cong \{0\}\times X\subset M_{\infty}$.
The manifold $X$ is assumed compact without boundary and is called the
\emph{cross-section} of~$M$. 
Let $t$ be a smooth real function on $M$ which coincides with the
$\bbrp$-coordinate on $M_{\infty}$, and is negative on the interior of $M_{0}$.
A metric $g$ on $M$ is called \emph{exponentially asymptotically cylindrical
(EAC) with rate $\delta>0$} if the functions
$e^{\delta t} \norm{\nabla^{k}_\infty(g-(dt^2+g_X)}$ on the end $M_{\infty}$
are bounded for all $k \geq 0$, where the point-wise norm $\|\cdot\|$ and the
Levi--Civita connection $\nabla_\infty$ are induced by some product Riemannian
metric $dt^2+g_X$ on $\bbrp\times X$. A Riemannian manifold (with cylindrical
ends) with an EAC metric will be called an \emph{EAC manifold}.

We can use $\nabla_\infty$ to define \emph{translation-invariant} tensor
fields on an EAC manifold $M$ as tensor fields whose restrictions to
$M_\infty$ are independent of~$t$.
A tensor field $s$ on $M$ is said to be \emph{exponentially asymptotic} with
rate $\delta > 0$ to a translation-invariant tensor $s_{\infty}$ on $M_\infty$
if $e^{\delta t} \norm{\nabla^{k}_\infty(s-s_{\infty})}$ are bounded
on $M_{\infty}$ for all $k \geq 0$.
A \gtstr{} is said to be \emph{EAC} if it is exponentially asymptotic to a
cylindrical \gtstr{} on $\bbrp\times X$. It is not difficult to check that
each EAC \gtstr{} $\varphi$ induces an EAC metric $g(\varphi)$.
The asymptotic limit of a torsion-free EAC \gtstr{} then defines a
\cystr{} on the cross-section $X$.

We shall need a topological criterion for an EAC \gtmfd\ to be irreducible.
\begin{thm}[{\cite[Theorem 3.8]{jn1}}]
\label{irredthm}
Let $(M^7,\varphi)$ be an EAC \gtmfd. Then the induced metric $g_\varphi$ has
full holonomy $G_2$ if and only if the fundamental group $\pi_1(M)$ is finite
and neither $M$ nor any double cover of~$M$ is homeomorphic to a cylinder
$\RE\times X^6$.
\end{thm}

\begin{cor}\label{irred}
Every simply-connected EAC \gtmfd{} with a single end (i.e.\ a connected
cross-section~$X$) is irreducible.
\end{cor}

\begin{rmk}
As every \gtmfd\ is Ricci-flat, the Cheeger--Gromoll line splitting theorem
\cite{cheeger71} implies that a connected EAC \gtmfd\ either has just one
end or two ends. In the latter case, the EAC \gtmfd\ is necessarily a cylinder
$\RE\times X$ with a product metric and cannot have full holonomy~$G_2$.
\end{rmk}

On an asymptotically cylindrical manifold $M$ it is useful to introduce
\emph{weighted Sobolev norms}. Let $E$ be a vector bundle on $M$ associated
to the tangent bundle, $k \geq 0$ and
$\delta \in \bbr$. We define the $\sobd{k}$-norm of a section $s$ of $E$
in terms of the usual Sobolev norm by
\begin{equation}
\norm{s}_{\sobd{k}} = \norm{e^{\delta t}s}_{\sob{k}} .
\end{equation}
Denote the space of sections of $E$ with finite $\sobd{k}$-norm
by $\sobd{k}(E)$.
Up to Lipschitz equivalence the weighted norms are independent of the choice
of asymptotically cylindrical metric, and of the choice of $t$ on the compact
piece $M_{0}$. In particular, the topological vector spaces $\sobd{k}(E)$ are
independent of these choices. As any asymptotically cylindrical manifold~$M$
clearly has bounded curvature and injectivity radius bounded away from zero,
the Sobolev embedding $L^2_k\subset C^r$ is still valid whenever $r<k-7/2$
\cite[\S~2.7]{aubin}.
It follows that $L^2_{k,\delta}$ consists of sections decaying (when
$\delta>0$) with all derivatives of order up to~$r$  at the rate
$O(e^{-\delta t})$ as $t\to\infty$.

An important property of the weighted norms is that elliptic linear operators
with asymptotically translation-invariant coefficients over~$M$ extend to
Fredholm operators between $\delta$-weighted spaces of sections, for `almost
all' choices of weight parameter $\delta$ \cite{lockhart87,lockhart85,MP}.
In particular, this can be applied to the Hodge Laplacian of an EAC metric
to deduce results analogous to Hodge theory for compact manifolds.
In this paper we shall require only a result about \emph{Hodge decomposition}.
Let $M^{n}$ be an EAC manifold with rate $\delta_{0}$ and cross-section $X$.
Abbreviate $\Lambda^{m}T^{*}M$ to $\Lambda^{m}$,
and let
\[ \sobd{k}[d\Lambda^{m-1}], \sobd{k}[d^{*}\Lambda^{m+1}]
 \subset \sobd{k}(\Lambda^{m}) \]
denote the subspaces of exact and coexact $\sobd{k}$ $m$-forms respectively.
Let $\harm^{m}_{+}$ denote the space of $L^{2}$ harmonic forms on $M$,
and $\harm^{m}_{\infty}$ the space of translation-invariant harmonic
forms on the product cylinder $X \times \bbr$.
If $\rho:M\to [0,1]$ is a smooth cut-off function supported on the cylindrical
ends $M_\infty$ of $M$ and such that $\rho\equiv 1$ in the region
$\{t>1\}\subset M$ then $\rho \harm^{m}_{\infty}$ can be identified with a
space of smooth $m$-forms on $M$.
Suppose that $0 < \delta < \delta_{0}$ and that $\delta^{2}$ is smaller than
any positive eigenvalue of the Hodge Laplacian on $\oplus_m\Lambda^mT^*X$
for the asymptotic limit metric $g_X$ on~$X$. Then the elements of
$\harm^{m}_{+}$ are smooth and decay exponentially with rate $\delta$
\cite{MP}. 

\begin{thm}[cf.~{\cite[p.~328]{jn1}}]
\label{hodgedecompthm}
In the notation above, there is an \ltwoorth{} direct sum decomposition
\begin{equation}
\label{noncpthodgeeq}
\sobd{k}(\Lambda^{m}) = \harm^{m}_{+} \oplus \sobd{k}[d\Lambda^{m-1}]
\oplus \sobd{k}[d^{*}\Lambda^{m+1}] .
\end{equation}
Furthermore, any element of $\sobd{k}[d\Lambda^{m-1}]$ can be written as
$d\phi$, for some coexact form
$\phi \in \sobd{k+1}(\Lambda^{m-1}) \oplus \rho\harm^{m-1}_{\infty}$.
\end{thm}

\section{Existence of EAC torsion-free $G_2$-structures}
\label{eacperturbsec}

We shall construct EAC manifolds with holonomy exactly $G_{2}$ by modifying
Joyce's construction of compact \gtmfd s. To this end, we shall obtain
a one-parameter family of \gtstr s with `small' torsion on a manifold with
cylindrical end. More precisely, this family will satisfy the hypotheses of
the following theorem, the main result of this section, which is an EAC
version of~\cite[Theorem~11.6.1]{joyce00}.

\begin{thm}
\label{eacperturbthm}
Let  $\mu, \nu, \lambda$ positive constants.
Then there exist positive constants $\kappa, K$ such that whenever
$0 < s < \kappa$ the following is true.

Let $M$ be a 7-manifold with cylindrical end $M_\infty$ and cross-section
$X^{6}$, and suppose that a closed stable 3-form $\tv$ defines on~$M$ a
\gtstr{} which is cylindrical and torsion-free on~$M_\infty$.
Suppose that $\psi$ is a smooth compactly supported $3$-form on $M$ satisfying
$d^{*}\psi = d^{*}\tv$, and let $r(\tv)$ and $R(\tv)$ be the injectivity
radius and Riemannian curvature of the EAC metric $g_{\tv}$ on~$M$. If
{
\renewcommand{\theenumi}{\alphenumi}
\begin{enumerate}
\item \label{psiestitem}
\begin{equation}
\label{psiesteq}
\norm{\psi}_{L^{2}} < \lambda s^{4},\quad
\norm{\psi}_{C^{0}} < \lambda s^{1/2},\quad
\norm{d^{*}\psi}_{L^{14}} < \lambda,
\end{equation}
\item $r(\tv) > \mu s$,
\item \label{curvestitem} $\norm{R(\tv)}_{C^{0}} < \nu s^{-2}$,
\end{enumerate}
}
then there is a smooth exact $3$-form $d\eta$ on $M$, exponentially decaying
with all derivatives as $t\to\infty$, such that
\begin{equation}
\label{eacdetasmalleq}
\norm{d\eta}_{L^{2}} < K s^{4}, \;\:
\norm{d\eta}_{C^{0}} < K s^{1/2}, \;\:
\norm{\nabla d\eta}_{L^{14}} < K,
\end{equation}
and $\varphi = \tv + d\eta$ is a torsion-free \gtstr.
\end{thm}

\begin{rmk}
\label{cptperturbrmk}
The difference between Theorem~\ref{eacperturbthm}
and~\cite[Theorem~11.6.1]{joyce00} is that $M$ is now non-compact with a
cylindrical end and we made appropriate assumptions on $\tv$, $\psi$ away
from a compact piece of~$M$ and are claiming an EAC property of the
resulting~$\varphi$.
On the other hand, formally, taking the cross-section $X^6$ to be empty
(hence $M$ being compact) recovers the statement of
\cite[Theorem~11.6.1]{joyce00}. 
\end{rmk}

\begin{rmk}
The fact that $d\eta$ is {\em exponentially} decaying is more important than
its precise rate of decay. We shall need to choose the rate
$\delta > 0$ so that $\delta^{2}$ is smaller than any non-zero
eigenvalue of the Hodge Laplacian on $X$.
It should be easy to modify the proof of the theorem to allow $\tv$ to be EAC
and $\psi$ to be exponentially decaying. In that case one would also need
$\delta$ to be smaller than the decay rates of $\tv$ and~$\psi$.
\end{rmk}

We wish to find an exact exponentially asymptotically decaying
$3$-form $d\eta$ such that $\tv + d\eta$ is torsion-free.
First we show that for $\tv + d\eta$ to
be torsion-free it suffices to show that $\eta$ is a solution of a certain
non-linear elliptic equation, which was also used by Joyce~\cite{joyce00} in
the compact case, and find a solution for this equation by a
contraction-mapping argument. The details of this are complicated, but largely
carry over from argument for the compact case worked out
in~\cite[Ch.11]{joyce00}. We initially obtain, adapting the method
of~\cite[Ch.11]{joyce00}, a closed 3-form $\chi$, so that $\phi+\chi$ is a
torsion-free \gtstr, and use elliptic regularity to show that the solution
$\chi$ is smooth and uniformly decaying along the end $M_\infty$ as
$t\to\infty$. Then, and this is an additional argument required for an EAC
manifold, we prove that the solution decays {\em exponentially}. This also
ensures that $\chi$ is exact, which will complete the proof of
Theorem~\ref{eacperturbthm}.

\subsection{Contraction mapping argument}

The proposition below is an asymptotically cylindrical version of
\cite[Theorem $10.3.7$]{joyce00}.

\begin{prop}
\label{suffprop}
There is an absolute constant $\epsilon_{1}>0$ such that the following holds.
Let $M^{7}$ be an EAC manifold, $\tv$ a closed EAC \gtstr{} on $M$ and $\psi$
an exponentially decaying $3$-form such that
$\norm{\psi}_{C^{0}} < \epsilon_{1}$ and $d^{*}\psi = d^{*}\tv$.
Suppose that $\eta$ is $2$-form 
asymptotic to a translation-invariant harmonic form, and that
$\norm{d\eta}_{C^{0}} < \epsilon_{1}$. Suppose further that
\begin{equation}
\label{suff2eq}
\triangle \eta = d^{*}\psi + d^{*}(f\psi) + *dF(d\eta) ,
\end{equation}
where the function $f$ is the point-wise inner product
$\third \inner{d\eta, \tv}$ and $F$ denotes
the quadratic and higher order parts, at $\tv$, of the non-linear fibre-wise
map $\Theta : \varphi \mapsto *_{\varphi}\varphi$ from \gtstr s to 4-forms.
Then $\tv + d\eta$ is a torsion-free EAC \gtstr{} on~$M$.
\end{prop}

\begin{proof}
The proof for the compact case in \cite{joyce00} relies on integrating by
parts.
It is easy to check that, in the asymptotically cylindrical setting,
the necessary integrals still converge provided that $\eta$ is bounded
and $d\eta$ decays, so we can still use
(\ref{suff2eq}) as a sufficient condition for the torsion to vanish.
\end{proof}

A key part in the proof of the existence of solutions for~\eqref{suff2eq} on a
compact 7-manifold is the contraction-mapping argument
\cite[Propn.~11.8.1]{joyce00}. We observe that it can easily be adapted to the
EAC case.

\begin{prop}
\label{estprop}
Let $(\Omega,\omega)$ be a \cystr{} on a compact manifold $X^{6}$
and $\mu, \nu, \lambda$ be positive constants.
Then there exist positive constants $\kappa, K, C_{1}$
such that whenever $0 < s < \kappa$ the following is true.

Let $M^{7}$ be a manifold with cylindrical end and cross-section $X$,
and $\tv$ a closed EAC \gtstr{} on $M$ with asymptotic
limit $\Omega + dt \wedge \omega$. Suppose that $\psi$ is a smooth
exponentially decaying $3$-form on $M$
satisfying $d^{*}\psi = d^{*}\tv$, and that
{
\renewcommand{\theenumi}{\alphenumi}
\begin{enumerate}
\item \label{lambdaitem} $\norm{\psi}_{L^{2}} < \lambda s^{4}$,
$\norm{\psi}_{C^{0}} < \lambda s^{1/2}$,
$\norm{d^{*}\psi}_{L^{14}} < \lambda$,
\item the injectivity radius is $> \mu s$,
\item the Riemannian curvature $R$ satisfies $\norm{R}_{C^{0}} < \nu s^{-2}$.
\end{enumerate}
}
Then there is a sequence $d\eta_{j}$ of smooth exponentially decaying
exact $3$-forms with $d\eta_{0} = 0$ satisfying the equation
\begin{equation}
\label{inducteq}
\triangle \eta_{j} = d^{*}\psi + d^{*}(f_{j-1}\psi) + *dF(d\eta_{j-1}),
\end{equation}
where $f_{j} = \third \inner{d\eta_j, \tv}$ for each $j > 0$.
The solutions satisfy the inequalities
\begin{enumerate}
\item \label{l2item}
$\norm{d\eta_{j}}_{L^{2}} < 2\lambda s^{4}$,
\item \label{l14item}
$\norm{\nabla d\eta_{j}}_{L^{14}} < 4C_{1}\lambda$,
\item \label{c0item}
$\norm{d\eta_{j}}_{C^{0}} < K s^{1/2}$,
\item \label{l2bitem}
$\norm{d\eta_{j+1} - d\eta_{j}}_{L^{2}} < 2^{-j}\lambda s^{4}$,
\item \label{l14bitem}
$\norm{\nabla (d\eta_{j+1} - d\eta_{j})}_{L^{14}} < 4 \cdot 2^{-j}C_{1}\lambda$,
\item \label{c0bitem}
$\norm{d\eta_{j+1} - d\eta_{j}}_{C^{0}} < 2^{-j}K s^{1/2}$.
\end{enumerate}
\end{prop}

\begin{proof}
The existence of the sequence $d\eta_j$ and the inequalities
\ref{l2item}--\ref{c0bitem} are proved inductively.
Take $\delta > 0$ smaller than the decay rates of $\tv$ and $\psi$
such that $\delta^{2}$ is smaller than any positive
eigenvalue of the Hodge Laplacian on~$X$, and let
$\rho$ be a cut-off function for the cylinder on $M$.
If $d\eta_{j-1}$ exists and satisfies the uniform estimate \ref{c0item}
then $F(d\eta_{j-1})$ is well-defined, and
the RHS of (\ref{inducteq}) is $d^{*}$ of a $3$-form that
decays with exponential rate $\delta$.
The EAC Hodge decomposition Theorem \ref{hodgedecompthm}
implies that there is a unique coexact solution
$\eta_{j} \in \sobd{k}(\Lambda^{2}) \oplus \rho \harm^{2}_{\infty}$
for all $k \geq 2$.

The induction step for the inequalities is proved
using exactly the same argument as in \cite[Propn.~11.8.1]{joyce00}.
\ref{l2item} and \ref{l2bitem} are proved using an integration by parts
argument, and since each $d\eta_{j}$ decays exponentially this is still 
justified when $M$ has cylindrical ends.

\ref{l14item}, \ref{c0item}, \ref{l14bitem} and \ref{c0bitem} are proved using
interior estimates, which do not require compactness.
\end{proof}

It follows that if $s$ is small, then $d\eta_{j}$ is a Cauchy sequence, in
each of the norms $L^2$, $L^{14}_1$ and $C^0$, and has a limit $\chi$ with
\begin{equation}
\label{chismalleq}
\norm{\chi}_{L^{2}} < K s^{4}, \;\:
\norm{\chi}_{C^{0}} < K s^{1/2}, \;\:
\norm{\nabla \chi}_{L^{14}} < K,
\end{equation}
for some $K > 0$. The form $\chi$ is closed, \ltwoorth{} to the space of
decaying harmonic forms $\harm^{3}_{+}$ and satisfies the equation
\begin{equation}
\label{chi1eq}
d^{*}\chi = d^{*}\psi + d^{*}(f\psi) + *dF(\chi) ,
\end{equation}
where $f = \third \inner{\chi, \tv}$.
We do not know \emph{a priori} that $\chi$ is the exterior derivative of a
bounded form, so we cannot yet apply Proposition \ref{suffprop} to show that
$\tv + \chi$ is torsion-free.

\subsection{Regularity}
 
We first show by elliptic regularity that $\chi$ is smooth and uniformly
decaying.

\begin{prop}
If $s$ is sufficiently small then $\chi \in \sobf{k}(\Lambda^{3})$ for all
$k \geq 0$.
\end{prop}

\begin{proof}
Since $F(\chi)$ depends only point-wise on $\chi$ and is of quadratic order
we can write
\begin{equation}
\label{expandfeq}
*dF(\chi) = P(\chi, \nabla \chi) + Q(\chi) ,
\end{equation}
where $P(u, v)$ is linear in $v$ and smooth of linear order in $u$, while
$Q(u)$ is smooth of quadratic order in $u$ for $u$ small. We can then rephrase
(\ref{chi1eq}) as stating that $\beta = \chi$ is a solution of
\begin{equation}
\label{chi2eq}
\begin{aligned}
d^{*}\beta - P(\chi, \beta) - d^{*}(f(\beta)\psi) &= d^{*}\psi + Q(\chi) ,
\\
d \beta &= 0 ,
\end{aligned}
\end{equation}
where $f(\beta) = \third \inner{\beta,\tv}$. The LHS is a linear
partial differential operator acting on $\beta$. Its symbol depends
on $\chi$ and $\psi$, but not on their derivatives.
By taking $s$ small we can ensure that $\chi$ and $\psi$ are both
small in the uniform norm (see (\ref{chismalleq}) and
hypothesis~\ref{lambdaitem} 
in Proposition \ref{estprop}) so that the equation is elliptic.

Now suppose that $\chi$ has regularity $\sobf{k}$. Then so do the coefficients
and the RHS of (\ref{chi2eq}). Because
$\beta = \chi \in \sobf{1}(\Lambda^{3})$ is a solution of (\ref{chi2eq}),
standard interior estimates
(see Morrey \cite[Theorems 6.2.5 and 6.2.6]{morrey66})
imply that it must have regularity $\sobf{k+1}$ locally.
Moreover, because the metric is asymptotically cylindrical the local
bounds are actually uniform, so in fact
$\chi$ is globally $\sobf{k+1}$. The result follows by induction on $k$.
\end{proof}

In the next result and in \S\ref{exp.dec}, we interchangeably consider $\chi$
on the cylindrical end $M_{\infty}=\bbrp\times X$ as a family of sections
over~$X$ depending on a real parameter~$t$.
\begin{cor}
\label{uniformcor}
If $s$ is sufficiently small then on the cylindrical end of $M$ the form
$\chi$ decays, with all derivatives, uniformly on~$X$ as $t \to \infty$.
\end{cor}

\begin{proof}
Because $M$ is EAC, standard Sobolev embedding results imply that
we can pick $r > 0$ such that $M$ is covered
by balls $B(x_{i}, r)$ with the following property:
\[ \norm{\chi|_{B(x_{i}, r)}}_{C^{k}} <
C \norm{\chi|_{B(x_{i}, 2r)}}_{L^{14}_{k+1}} , \]
where the constant $C > 0$ is independent of $x_{i} \in M$.
If we ensure that each point of $M$ is contained in no more than $N$ of the
balls $B(x_{i}, 2r)$ then
\[ \sum_{i} \norm{\chi|_{B(x_{i}, r)}}_{C^{k}}^{14} <
N C^{14} \norm{\chi}_{L^{14}_{k+1}}^{14} . \]
As the sum is convergent the terms tend to $0$, i.e. the $k$-th derivatives
of $\chi$ decay uniformly.
\end{proof}

\subsection{Exponential decay}
\label{exp.dec}

To complete the proof of Theorem \ref{eacperturbthm} it remains to
prove that the rate of decay of $\chi$ is exponential.
Then $\chi = d\eta$ for some exponentially asymptotically translation-invariant
$\eta$ by the Hodge decomposition Theorem \ref{hodgedecompthm}, since $\chi$ is
closed and \ltwoorth{} to the decaying harmonic forms $\harm^{3}_{+}$.
Proposition \ref{suffprop} then implies that $\tv + d\eta$ is torsion-free,
so that $d\eta$ has all the desired properties.

By hypothesis, $\tv$ is exactly cylindrical on the cylindrical end
$M_\infty=\{t \geq 0\}$ of~$M$, and $\psi$ is supported in the compact piece
$M_0=\{t \leq 0\}$.
Thus on the cylindrical end the equation (\ref{chi1eq}) for $\chi$ simplifies
to
\begin{equation}
\label{chi3eq}
d^{*}\chi = *dF(\chi) .
\end{equation}
On the cylindrical end $t > 0$ we can write
\begin{align*}
\chi &= \sigma + dt \wedge \tau , \\
F(\chi) &= \beta + dt \wedge \gamma ,
\end{align*}
where $\tau \in \Omega^{2}(X)$, $\sigma, \gamma \in \Omega^{3}(X)$ and
$\beta \in \Omega^{4}(X)$ are forms on the cross-section $X$ depending on the
parameter $t$. Let $\dsubx$ denote the exterior derivative on $X$.
Then the conditions $d\chi = 0$ and (\ref{chi3eq}) are
equivalent to
\begin{subequations}
\begin{align}
\label{cylchi1eq} \dsubx\sigma &= 0 , \\
\label{cylchi2eq} \contrat \sigma &= \dsubx\tau , \\
\label{cylchi3eq} \dsubx {*} \tau &= -\dsubx\beta , \\
\label{cylchi4eq} \contrat {*} \tau &=
- \dsubx {*} \sigma - \contrat \beta + \dsubx\gamma .
\end{align}
\end{subequations}
(\ref{cylchi2eq}) implies that $\sigma(t_{1}) - \sigma(t_{2})$ is exact for
any $t_{1}, t_{2} > 0$. Since the exact forms form a closed subspace of
the space of $3$-forms on $X$ (in the $L^{2}$ norm) and $\sigma \to 0$ as
$t \to \infty$ it follows that $\sigma$ is exact for all $t > 0$.
Similarly (\ref{cylchi4eq}) implies that $*\tau - \beta$ is exact for
all $t > 0$.
(The equations (\ref{cylchi1eq}) and (\ref{cylchi3eq}) are thus redundant.)
The path $(\sigma, \tau)$ is therefore constrained to lie in the space
\begin{equation*}
\calf = \{ (\sigma, \tau) \in \dsubx\sob{1}(\Lambda^{2}T^{*}X) \times
L^{2}(\Lambda^{2}T^{*}X) : *\tau - \beta \textrm{ is exact} \} .
\end{equation*}

\begin{rmk}
We have not assumed that $\chi$ is $L^{1}$ on $M$.
\end{rmk}

$\beta$ is a function of $\sigma$ and $\tau$, and it is of quadratic order.
The implicit function theorem applies to show that if we replace $\calf$ with
a small neighbourhood of $0$ then it is a Banach manifold with tangent space
\begin{equation*}
T_{0}\calf = B =  \dsubx\sob{1}(\Lambda^{2}T^{*}X) \times
\dcsubx\sob{1}(\Lambda^{3}T^{*}X) .
\end{equation*}

We can now interpret (\ref{cylchi2eq}) and (\ref{cylchi4eq}) as a flow
on $\calf$, or equivalently near the origin in $B$.
By the chain rule we can write $\contrat \beta$ as
\begin{equation*}
\contrat \beta = A_{2}\left(\contrat\tau\right) +
	A_{3}\left(\contrat\sigma\right) + \beta',
\end{equation*}
where $A_{m}$ is a linear map from $\Lambda^{m}T^{*}X$ to $\Lambda^{4}T^{*}X$,
determined point-wise by $\sigma$ and $\tau$ and of linear order, while
$\beta'$ is a $4$-form determined point-wise by $\sigma$ and $\tau$ and of
quadratic order. In particular, for large $t$ the norm of $A_{2}$ is small,
and (\ref{cylchi2eq}) and (\ref{cylchi4eq}) are equivalent to
\begin{equation}
\label{floweq}
\begin{aligned}
\contrat \sigma &= \dsubx\tau , \\
\contrat\tau  &= (id + *A_{2})^{-1}
(\dcsubx \sigma - *A_{3}\dsubx\tau - *\beta' + *\dsubx\gamma) .
\end{aligned}
\end{equation}
The origin is a stationary point for the flow, and the linearisation of the
flow near the origin is given by the (unbounded) linear operator
$L = \tbtmatrix{\dcsubx}{0}{0}{\dsubx}$ on $B$.
Because $L$ is formally self-adjoint $B$ has an orthonormal basis of
eigenvectors. Also, $L$ is injective on $B$, so $B$ can be written as
a direct sum of subspaces with positive and negative eigenvalues,
\[ B = B_{+} \oplus B_{-} . \]
Then $\{e^{\mp t L} : t \geq 0 \}$ defines a continuous semi-group of bounded
operators on $B_{\pm}$.
If we let $\mu$ denote the smallest absolute value of the eigenvalues of $L$
then $e^{t\mu}e^{\mp t L}$ is uniformly bounded on $B_{\pm}$ for $t \geq 0$, so
the origin is a hyperbolic fixed point.
By analogy with the finite-dimensional flows, we expect that any solution of
(\ref{cylchi2eq}) and (\ref{cylchi4eq}) approaching the origin must do so at
an exponential rate.

A similar problem of exponential convergence for an infinite-dimensional flow
is considered by Mrowka, Morgan and Ruberman \cite[Lemma 5.4.1]{mmr94}.
Their problem is more general in that the linearisation of their flow has
non-trivial kernel, so that they need to consider convergence to a `centre
manifold' rather than to a well-behaved isolated fixed point. As a simple
special case we can prove the $L^{2}$ exponential decay for $\chi$.

\begin{prop}
Let $\delta > 0$ such that $\delta^{2}$ is smaller than any positive eigenvalue
of the Hodge Laplacian on $X$. Then $\chi$ is $L^{2}_{\delta}$.
\end{prop}

\begin{proof}
Identify $\calf$ with a neighbourhood of the origin in the tangent space $B$,
and let $x$ be the path in $B$ corresponding to $(\sigma, \tau)$ in $\calf$.
Then (\ref{floweq}) transforms to a differential equation for $x$,
\begin{equation*}
\frac{dx}{dt} = Lx + Q(x) ,
\end{equation*}
where $L$ is the linearisation of (\ref{floweq}) as above,
and $Q$ is the remaining quadratic part.
Let $x = x_{+} + x_{-}$ with $x_{\pm} \in B_{\pm}$. 
If, as before, $\mu$ denotes the smallest absolute value of the eigenvalues
of $L$ then
\begin{equation*}
\lnorm{Lx_{+}} \geq \mu \lnorm{x_{+}}, \quad
\lnorm{Lx_{-}} \leq -\mu \lnorm{x_{-}} .
\end{equation*}
Applying the chain rule to the quadratic part of (\ref{floweq}) gives
\begin{equation*}
\lnorm{Q(x)} < O(\lnorm{x})\norm{x}_{\sob{1}} + O(\lnorm{x}^{2}).
\end{equation*}
By corollary \ref{uniformcor}, $x$ converges uniformly to $0$ with all
derivatives as $t \to \infty$. Therefore for any fixed $k > 0$ we can find
$t_{0}$ such that
\begin{equation*}
\lnorm{Q(x)} < k\lnorm{x}
\end{equation*}
for any $t > t_{0}$. As $\mu^{2}$ is an eigenvalue for the
Hodge Laplacian on $X$ we may fix $k$ so that $\mu - 2k > \delta$.

We thus obtain that for $t > t_{0}$
\begin{subequations}
\begin{gather}
\dddt \lnorm{x_{+}} \geq \phantom{-} \mu \lnorm{x_{+}} - k\lnorm{x}, \\
\label{xmeq}
\dddt \lnorm{x_{-}} \leq -\mu \lnorm{x_{-}} + k\lnorm{x} .
\end{gather}
\end{subequations}
In particular $\lnorm{x_{+}} - \lnorm{x_{-}}$ is an increasing function of $t$.
Because it converges to $0$ as $t \to \infty$,
\begin{equation*}
\lnorm{x_{+}} \leq \lnorm{x_{-}}
\end{equation*}
for all $t > t_{0}$. Substituting into (\ref{xmeq})
\begin{equation*}
\dddt \lnorm{x_{-}} \leq -\mu \lnorm{x_{-}} + 2k \lnorm{x_{-}} ,
\end{equation*}
so $\lnorm{x_{-}}$ is of order $e^{(-\mu+2k)t}$.
Hence so is $\lnorm{x}$, so $e^{\delta t}\chi$ is $L^{2}$-integrable on $M$.
\end{proof}

\begin{cor}
$\chi$ decays exponentially with rate $\delta$.
\end{cor}

\begin{proof}
We prove by induction that $\chi$ is $\sobd{k}$ for all $k \geq 0$.
Interior estimates for the elliptic operator $d + d^{*}$ on $M$ imply that
we can fix some $r > 0$ and cover the cylindrical part of $M$ with open balls
$U = B(x,r)$ such that
\begin{equation*}
\norm{\chi}_{\sob{k+1}(U)} <
C_{1} \left (\norm{d\chi}_{\sob{k}(U)} + \norm{d^{*}\chi}_{\sob{k}(U)} \right) +
C_{2}\norm{\chi}_{L^{2}(U)} .
\end{equation*}
The constants $C_{1}$ and $C_{2}$ depend on the local properties of the metric
and the volume of $U$. 
Since $M$ is EAC we can take the constants to be independent of $U$.
Recall that on the cylinder $d\chi = 0$ and $d^{*}\chi = *dF(\chi)$.
In view of the chain rule expression (\ref{expandfeq}) there is a constant
$C_{3} > 0$ such that
\begin{equation*}
\norm{dF(\chi)}_{\sob{k}(U)} < C_{3} \norm{\chi}_{C^{k}(U)}
\left(\norm{\nabla\chi}_{\sob{k}(U)} + \norm{\chi}_{\sob{k}(U)}\right) .
\end{equation*}
As $\chi$ decays uniformly we can ensure that
$\norm{\chi}_{C^{k}(U)} < 1/2C_{1}C_{3}$ by taking $U$ to be sufficiently far
along the cylindrical end. Then
\begin{equation*}
\norm{\chi}_{\sob{k+1}(U)} < \norm{\chi}_{\sob{k}(U)} +
2C_{2} \norm{\chi}_{L^{2}(U)} .
\end{equation*}
Hence $\chi$ is $\sobd{k}$ for all $k \geq 0$.
\end{proof}

This completes the proof of Theorem \ref{eacperturbthm}.

\section{Constructing an EAC $G_2$-manifold}
\label{exsec}

We shall obtain examples of torsion-free EAC \gtstr s by modifying
one of the compact 7-manifolds $M$ with holonomy $G_2$ constructed by Joyce
\cite{joyce00}. Our EAC \gtmfd s will arise in pairs via a decomposition of a
compact $M$ into two compact manifolds identified along their common boundary,
a 6-dimensional submanifold $X\subset M$,
\begin{subequations}\label{decomp}
\begin{equation}
M=M_{0,+}\cup_X M_{0,-}.
\end{equation}
A collar neighbourhood of the boundary of each $M_{0,\pm}$ is
diffeomorphic to $I\times X$, for an interval $I\subset\RE$. Define
\begin{equation}
M_{\pm}=M_{0,\pm}\cup_X (\bbrp\times X).
\end{equation}
\end{subequations}
It is on the manifolds $M_{\pm}$ with cylindrical ends that we shall construct
EAC \gtstr s satisfying the hypotheses of Theorem \ref{eacperturbthm}, such
that the resulting EAC \gtmfd s have holonomy~$G_{2}$. (Of course, $M_\pm$ is
homeomorphic to the interior of $M_{0,\pm}$.)

\subsection{Joyce's example of a compact irreducible $G_2$-manifold}
\label{simplesub}

In order to give examples of $M_{\pm}$ as above, we need to recall part
of the construction of a relatively uncomplicated example of a compact
$G_2$-manifold in~\cite[\S 12.2]{joyce00}. Consider the action on a torus
$T^{7}$ by the group $\Gamma \cong \bbz^{3}_{2}$ generated by
\begin{equation}
\label{simplegammaeq}
\begin{aligned}
\alpha &: (x_{1}, \ldots, x_{7}) \mapsto
(\phantom{-}x_{1}, \phantom{-}x_{2}, \phantom{-}x_{3}, -x_{4}, 
\phantom{\half}{-}x_{5}, \phantom{\half}{-}x_{6}, \phantom{\half}{-}x_{7}) , \\
\beta &: (x_{1}, \ldots, x_{7}) \mapsto
(\phantom{-}x_{1}, -x_{2}, -x_{3}, \phantom{-} x_{4},
\phantom{\half{-}} x_{5}, \half {-} x_{6}, \phantom{\half}{-}x_{7}) , \\
\gamma &: (x_{1}, \ldots, x_{7}) \mapsto
(-x_{1}, \phantom{-}x_{2}, -x_{3}, \phantom{-} x_{4},
\half{-}x_{5}, \phantom{\half{-}} x_{6}, \half{-}x_{7}) .
\end{aligned}
\end{equation}
These maps preserve the standard flat \gtstr{} on $T^{7}$
(cf. (\ref{g2formeq})), so $T^{7}/\Gamma$ is a flat compact
$G_{2}$-orbifold. It is simply-connected. 

The fixed point set of each of $\alpha, \beta$ and $\gamma$ consists of $16$
copies of $T^{3}$ and these are all disjoint. $\alpha\beta$, $\beta\gamma$,
$\gamma\alpha$ and $\alpha\beta\gamma$ act freely on $T^{7}$.
Furthermore $\gen{\beta, \gamma}$ acts freely on the set of $16$ $3$-tori
fixed by $\alpha$, so they map to $4$ copies of $T^{3}$ in the singular locus
of $T^{7}/\Gamma$. Similarly $\gen{\alpha,\gamma}$ and $\gen{\alpha,\beta}$
acts freely on the $16$ $3$-tori fixed by $\beta$ and $\gamma$, respectively.
Thus the singular locus of $T^{7}/\Gamma$ consists of $12$ 
disjoint copies of~$T^{3}$.

A neighbourhood of each component $T^{3}$ of the singular locus of
$T^{7}/\Gamma$ is diffeomorphic to $T^3\times\bbc^{2}/\{\pm 1\}$.
The blow-up of $\bbc^{2}/\{\pm 1\}$ at the origin resolves the singularity
giving a complex surface $Y$ biholomorphic to $T^{*}\CP^{1}$, with the
exceptional divisor corresponding to the zero section $\CP^{1}$. \label{EH}
The canonical bundle of~$Y$ is trivial and $Y$ has a family of
asymptotically locally Euclidean (ALE) Ricci-flat Kähler (hyper-Kähler)
metrics with holonomy $SU(2)$. These metrics may be defined via their Kähler
forms $i\p\bar\p f_s$, in the complex structure on~$Y$ induced by from
$T^{*}\CP^{1}$, where
\begin{equation}\label{EHmetr}
f_s=\sqrt{r^4+s^4}+2s^2\log s-s^2\log(\sqrt{r^4+s^4}+s^2),
\qquad
r^2=z_1\bar{z}_1+z_2\bar{z}_2,
\end{equation}
and $z_1,z_2$ are coordinates on $\CX^2$ and $s>0$ is a scale parameter.
The forms $i\p\bar\p f_s$ admit a smooth extension over the exceptional
divisor. Metrics induced by $f_s$ in~\eqref{EHmetr} are the well-known
Eguchi--Hanson metrics~\cite{EH},\cite[Ch.~7]{joyce00}.

It is known (and easy to check) that for each $\lambda>0$ the map $Y \to Y$
induced by $(z_1,z_2)\mapsto \lambda(z_1,z_2)$ pulls back $i\p\bar\p f_s$ to
$i\lambda^2\p\bar\p f_{\lambda s}$. In particular, $s$ is proportional to the
diameter of the exceptional divisor on~$Y$. Further, an important property of
the Eguchi--Hanson metrics is that the injectivity radius is proportional to
$s$ whereas the uniform norm of the curvature is proportional to $s^{-2}$.

As discussed in \S \ref{su2sub}, the product of an $SU(2)$-manifold and a
flat $3$-manifold has a `natural' torsion-free \gtstr{}~\eqref{g2su2}.
By replacing a neighbourhood of each singular $T^{3}$ in $T^{7}/\Gamma$ by
the product of $T^{3}$ and a neighbourhood $U\subset Y$ of the exceptional
divisor in the Eguchi-Hanson space one obtains a compact smooth manifold $M$.
Now $f_s$, for each $s>0$, is asymptotic to~$r^2$ as $r\to\infty$ and
$i\p\bar\p r^2$ is the Kähler form of
the flat Euclidean metric on $\bbc^{2}/\{\pm 1\}$.
It is therefore possible to smoothly interpolate between the torsion-free
\gtstr s on $T^{3} \times U$ corresponding to the Eguchi--Hanson metrics and
the flat \gtstr\ on $T^{7}/\Gamma$ away from a neighbourhood of the singular
locus, using a cut-off function in the gluing region. In this way, 
one obtains, for each small $s>0$, a closed stable 3-form, say $\phin_s$,
on~$M$, so that the induced \gtstr{} is torsion-free, except in the gluing
region. Altogether, according to~\cite[\S 11.5]{joyce00} the torsion $\phin_s$
is `small' in the sense that $d^{*}\phin_s=d^{*}\psi_s$ for some 3-forms
$\psi_s$ satisfying
\begin{equation}\label{small}
\norm{\psi_s}_{L^{2}} < \lambda' s^{4},\quad
\norm{\psi_s}_{C^{0}} < \lambda' s^{1/2},\quad
\norm{d^{*}\psi_s}_{L^{14}} < \lambda',
\end{equation}
for some constant $\lambda'$ independent of~$s$ (cf.~\eqref{psiesteq}).
By~\cite[Theorem~11.6.1]{joyce00} (cf. Remark \ref{cptperturbrmk}),
there is a constant $\kappa_M>0$, so that the 
\gtstr\ $\phin_s$ can be perturbed into a torsion-free \gtstr{}
\begin{equation}\label{joyce.g2}
\varphi_s=\phin_s+(\text{exact form})
\end{equation}
inducing a metric $g(\varphi_s)$ with holonomy~$G_2$ on~$M$ whenever
$0<s<\kappa_M$.

We also recall from \cite[\S 12.1]{joyce00} the technique for computing the
Betti numbers of the resolution $M$. This will be needed later when we
compute Betti numbers of the EAC \gtmfd s~$M_{\pm}$.

The cohomology of $T^{7}/\Gamma$ is just the $\Gamma$-invariant part of
the cohomology of $T^{7}$, so $b^{2}(T^{7}/\Gamma) = 0$ while
$b^{3}(T^{7}/\Gamma) = 7$.
For each of the $12$ copies of $T^{3}$ in the
singular locus we cut out a tubular neighbourhood, which deformation retracts
to~$T^{3}$, and glue in a piece of $T^{3} \times Y$, which
deformation retracts to $T^{3} \times \CP^{1}$.
Each of the operations increases the Betti numbers of $M$ by the difference
between the Betti numbers of $T^{3} \times Y$ and $T^{3}$.
This is justified using the long exact sequences for the cohomology of
$T^{7}/\Gamma$ relative to its singular locus and $M$ relative to the resolving
neighbourhoods.
Hence
\begin{equation*}
\begin{aligned}
b^{2}(M) &= 12 \cdot 1 = 12, \\
b^{3}(M) &= 7 + 12 \cdot 3 = 43 .
\end{aligned}
\end{equation*}

\subsection{An EAC \gtmfd}
\label{EACex}

We can let the group $\Gamma$ defined above act on $\bbr \times T^{6}$
instead of $T^{7}$, taking $x_{1}$ to be the coordinate on the $\bbr$-factor.
Then $(\bbr \times T^{6})/\Gamma$ is a flat $G_{2}$-orbifold with a single
end. We want to resolve it to an EAC \gtmfd.

The fixed point set of each of $\alpha$ and $\beta$ in $\bbr \times T^{6}$
consists of $16$ copies of $\bbr \times T^{2}$ and the fixed point set of
$\gamma$ consists of 8 copies of $T^3$. Resolving the singularities of
$(\bbr \times T^{6})/\Gamma$ arising from $\alpha,\beta$ by gluing in copies of
$\bbr \times T^{2} \times Y$ (along with resolving the $T^3$ singularities
arising from $\gamma$ as before) yields a 
smooth manifold $M_{+}$ with a single end (the cross-section $X$ of $M_+$
is a resolution of $T^6/\Gamma'$, where $\Gamma' \subset \Gamma$ is the
subgroup generated by $\alpha$ and~$\beta$).
However, the \gtstr{} defined by
naively adapting the method of the last subsection would introduce torsion in
a non-compact region, making it difficult to perturb to a torsion-free \gtstr.
To apply Theorem \ref{eacperturbthm} we need to
ensure that the \gtstr{} is exactly cylindrical and torsion-free on the
cylindrical end, so there may only be torsion in a compact region.
We shall get round this problem by performing the resolution in two steps,
and prove the following.

\begin{thm}
The manifold $M_+$ with cylindrical end and cross-section~$X$, as defined
in the beginning of this subsection, has an EAC metric with holonomy equal
to~$G_2$. The asymptotic limit metric on $X$ has holonomy equal to~$SU(3)$.
\label{1st.EAC-G2}
\end{thm}

Before giving the details of the proof of Theorem \ref{1st.EAC-G2}, let us
change perspective slightly and explain how the latter 7-manifold $M_+$ arises
in the setting~\eqref{decomp}, with $M$ the compact 7-manifold discussed in
\S\ref{simplesub}.
The image of a hypersurface $T^{6} \subset T^{7}$ defined by $x_{1} = \quart$
is a hypersurface orbifold $X_{0}$ which divides $T^{7}/\Gamma$ into two open
connected regions. In fact, $X_{0}$ is precisely $T^{6}/\Gamma'$, as $\Gamma'$
is the subgroup that acts trivially on the $x_{1}$ factor in $T^{7}$.
Each component of $(T^{7}/\Gamma)\setminus X_0$ is the interior of a compact
orbifold with boundary $X_{0}$ and we can attach product cylinders
$\bbr_{>0} \times X_{0}$ to form orbifolds with a cylindrical end.
One of these (the one containing the image of $x_{1} = 0$) corresponds
naturally to $(\bbr \times T^{6})/\Gamma$.

Now, $M_+$ is well-defined as a resolution of singularities of this $(\bbr
\times T^{6})/\Gamma$ as described above and $M_-$ is defined similarly by
starting from the other component of $T^{7}/\Gamma\setminus X_0$.

\begin{rmk}
In this particular example, the two EAC halves $M_\pm$ will be
isometric, the isometry being induced from an involution on $T^{7}/\Gamma$,
$$
(x_{1}, \ldots, x_{7}) \mapsto
(x_{1} + \half, x_{2}, x_{3}, x_{4}, x_{5}, x_{6}, x_{7}),
$$
which swaps the two components of $(T^{7}/\Gamma) \setminus X_{0}$. The
restriction to $X_{0}$ induces an anti-holomorphic isometry on its
resolution~$X$.
\label{isometry}
\end{rmk}

We now state a technical result from which Theorem~\ref{1st.EAC-G2} will
follow. 

\begin{prop}\label{intermed.phi}
Let $M$ be a smooth compact 7-manifold obtained by resolving singularities of
$T^7/\Gamma$, as defined in \S\ref{simplesub}.
There exists a constant $\kappa'>0$, such that for each $s$ with $0<s<\kappa'$,
there is a closed stable
$\tv_s\in\Omega^3(M)$ with the following properties:
\begin{enumerate}
\item \label{neckitem}
There is a Calabi--Yau structure $(\Omega, \omega)$ on a $6$-manifold $X$
and an interval $I = (-\epsilon,\epsilon)$ such
that $M$ has an open subset $N \cong X \times I$ with
\begin{equation}
\label{neckeq}
\tv_s|_{N} = \Omega + dt \wedge \omega ,
\end{equation}
and $N$ retracts to~$X$ and the complement of $N$ in $M$ has exactly two
connected components (diffeomorphic to the components of $M{\setminus}X$). 
\item There is a smooth $3$-form $\psi_s$ such that
$d^{*}\psi_s = d^{*}\tv_s$, satisfying the estimates \eqref{psiesteq},
with $\lambda > 0$ independent of $s$.
\item \label{psineckitem}
$\psi_s$ vanishes on $N$.
\item \label{exactitem}
The 3-form $\tv_s-\phin_s$ is exact, where $\phin_s$ is the
\gtstr{} on $M$ defined in~\S\ref{simplesub}.
\end{enumerate}
\end{prop}

We can think of $\tv_s$ as an `intermediate' perturbation of $\phin_s$.
Instead of perturbing away all the torsion in one go, like
in~\S\ref{simplesub}, 
we settle for eliminating the torsion from the neck region~$N$, while keeping
it controlled elsewhere. What we gain is that $\tv_s$ is a product \gtstr{}
on $N$. We can therefore cut $M$ into two halves along the hypersurface
$X \times \{0\} \subset N$, and attach a copy of $X \times [0, \infty)$ to each
half to form cylindrical-end manifolds $M_{\gi}$ with EAC \gtstr s
$\tv_{s,\gi}=\tv_s|_{M_{\gi}}$.
The properties (i)-(iii) achieved in Proposition~\ref{intermed.phi}
ensure that Theorem \ref{eacperturbthm} then applies to each of $M_{\gi}$,
giving $0<\kappa \leq \kappa'$ such that $\tv_{s,\gi}$ can be
perturbed to torsion-free \gtstr s
$$
\varphi_{s,\gi}=\tv_{s,\gi}+d\eta_{s,\gi},
$$
whenever $0<s<\kappa$.

The orbifold $(\bbr \times T^{6})/\Gamma$ is simply-connected, and so is the
resolution $M_+$. Therefore, any torsion-free \gtstr\ on $M_+$ induces a
metric with full holonomy $G_2$ by Corollary~\ref{irred}, which proves
Theorem \ref{1st.EAC-G2} assuming Proposition~\ref{intermed.phi}.

\begin{rmk}
\label{cutrmk}
This construction of the EAC \gtstr s with small torsion is only superficially
different from the description given before the statement of Theorem
\ref{1st.EAC-G2}. That is, the choice of whether we cut the manifold in half
and attach cylinders before or after resolving the singularities of the neck
is not particularly important.
The convenience of going with the latter choice in the proof is that it allows
us to do most of the technical work on compact manifolds. Another advantage is
that then it is better illuminated that we obtain a pair
of torsion-free EAC \gtmfd s whose asymptotic models are isomorphic.
One can apply to this pair of \gtstr s the gluing theorem from 
\cite[\S 5]{kovalev03} and obtain a \gtstr{} on the generalized connected sum
of $M_\pm$ joined at their ends, giving a compact \gtmfd{} with a long neck.
This connected sum is, of course, diffeomorphic to the compact \gtmfd{} $M$
as obtained by resolving singularities $T^{7}/\Gamma$ directly as in
\S\ref{simplesub}. Considering the $G_2$-metrics one may intuitively think of
the EAC halves $M_\pm$ being obtained by `pulling $M$ apart'. This will be
made more precise in~\S\ref{pullsec}, where the clause \ref{exactitem} of
Proposition~\ref{intermed.phi} will be important.
\end{rmk}

\subsection{Proof of Proposition \ref{intermed.phi}}
\label{four.three}
We find the desired cylindrical-neck \gtstr{} $\tv_s$ on the resolution
$M$ of $T^7/\Gamma$ by performing the resolution in two stages.
The group $\Gamma$ preserves the product decomposition $T^7 = S^1 \times T^6$,
where the $S^1$ factor corresponds to the $x_1$ coordinate.
Let $\Gamma' \subset \Gamma$ be the subgroup generated by $\alpha$ and $\beta$;
notice that $\Ga'$ acts on $T^{6}$ (and fixes the $S^1$-factor).
Define $\Psi=\Gamma/\Gamma'$. 
Here is the strategy of our proof:
{\renewcommand{\theenumi}{\textup{\arabic{enumi}}}
\renewcommand{\labelenumi}{\theenumi.}
\begin{enumerate}
\item \label{firstresitem}
Resolve the singularities of $T^7/\Gamma'$ using Eguchi-Hanson hyper-Kähler
spaces as described in \S \ref{simplesub} to form a compact manifold
$M'\cong S^1 \times X^6$ equipped with a family of $\Psi$-invariant
\gtstr{}s $\tv'_s$ with small torsion.
Perturb $\tv'_s$ to a torsion-free \mbox{$\Psi$-invariant} product \gtstr{}
$\varphi'_s$ on~$M'$.

\item The \gtstr\ $\varphi'_s$ is not flat near the fixed point set $F$
of $\Psi$ acting on~$M'$. We perturb $\varphi'_s$ by adding an exact 3-form
supported near~$F$, so that the resulting \gtstr\  on $M'$ interpolates
between the flat structure near $F$ and $\varphi'_s$ away from $F$. The
torsion introduced by the latter perturbation 3-form is controlled by estimates
similar to~\eqref{psiesteq}. 
Furthermore, the interpolating \gtstr\ is $\Psi$-invariant and descends
to the orbifold~$M'/\Psi$ (see Figure~\ref{thepic}).

\item Resolve the singularities of $M'/\Psi$, using the same Eguchi-Hanson
hyper-Kähler structures as in the construction of $\phin_s$ in
\S \ref{simplesub} (in particular, they have \emph{the same scale
parameter}~$s$ as in step~\ref{firstresitem}) and construct the \gtstr{}
$\tv_s$ on the compact manifold~$M$. Finally, check that the difference
$\tv_s - \phin_s$ is essentially the exact form added in step 2.
\end{enumerate}}

Our first step is entirely analogous to the construction of $\phin_s$ outlined
in~\S\ref{simplesub}, but this time we resolve the singularities of the
orbifold $(S^{1} \times T^{6})/\Gamma'$ rather than $T^7/\Gamma$.
This gives a compact 7-manifold $M'$ with a family of closed $S^{1}$-invariant
3-forms, say $\tv'_s$, inducing \gtstr s with small torsion in the sense
of~\eqref{small}.
Then, as noted in Remark \ref{cptperturbrmk}, we can apply
\cite[Theorem~11.6.1]{joyce00} and obtain a $\kappa' > 0$, such
that $\tv'_s$ admits a perturbation to a torsion-free \gtstr{}
\begin{equation}\label{perturb}
\varphi'_s=\tv'_s+d\eta'_s,
\end{equation}
for $0 < s < \kappa'$.
The correction term satisfies
\begin{equation}\label{tor}
\norm{d\eta_s}_{L^{2}} < K' s^{4}, \;\:
\norm{d\eta_s}_{C^{0}} < K' s^{1/2}, \;\:
\norm{\nabla d\eta_s}_{L^{14}} < K',
\end{equation}
with some constant $K'$ independent of~$s$ (cf.~\eqref{eacdetasmalleq}).

Clearly, there is a diffeomorphism
$$
M' \simeq S^{1} \times X,
$$
where $X$ denotes a blow-up of the complex orbifold $T^6/\Gamma'$.
Since $\tv'_s$ is $S^1$-invariant, so is~$\varphi'_s$; in fact, more is true.
The lemma below can be thought of as a simple version of the Cheeger--Gromoll
line splitting theorem (cf. \cite{cheeger71})
and ensures that $\varphi'_s$ is a product \gtstr{} determined by some
\cystr\ on~$X$ and some diffeomorphism $M' \cong S^1 \times X$ (but not
necessarily the same one as for $\tv'_s$).

\begin{lem}[{cf. Chan \cite[p. 15]{chan06}}]
\label{chanlem}
Let $T^{m}$ be a torus and $X$ a compact manifold with \mbox{$b^{1}(X) = 0$.}
If $g$ is a Ricci-flat metric on $T^{n} \times X$ that is invariant under
translations of the torus factor then there is a function
$f : X \to \bbr^{n}$ such that the graph diffeomorphism
\[ T^{n} \times X \to T^{n} \times X, \;\: (t,x) \mapsto (t + f(x), x) \]
pulls $g$ back to a product metric.
\end{lem}

\begin{proof}[Sketch proof.]
Let $\contrax{1}, \ldots, \contrax{n}$ be the coordinate vector fields on
$T^{n}$ and $\alpha_{i} = \bigl(\contrax{i}\bigr)^{\flat}$.
Each $\contrax{i}$ is a Killing vector field on a Ricci-flat manifold,
so the 1-forms $\alpha_{i}$ are harmonic. 
Since $b^{1}(X) = 0$ the closed forms $\alpha_{i}$ are exact.
Define $f : X \to \bbr^{n}$ by choosing $f_{i}$ such that
$\alpha_{i} = -df_{i}$.
\end{proof}

The following commutative diagram shows the relation between
$M'\simeq S^{1}\times X$ and $M$ in the resolution of singularities
and will be useful for keeping track of the construction of the desired
$\tv_s$ on~$M$ from \gtstr s $\tv'_s$ and $\varphi'_s$ on~$M'$.
\begin{equation}\label{res}
\begin{array}[c]{c}
{\xymatrix{
&M \ar[d] \ar@/^2pc/@{-->}[dd]\\
M' \ar[d] \ar[r]^{[\Psi]} &M'/\Psi \ar[d] \\
S^1\times(T^6/\Ga') \ar[r]^-{[\Psi]} &T^7/\Ga}}
\end{array}
\end{equation}
Here we used $[\Psi]$ to denote the quotient maps for the actions of
$\Psi=\Gamma/\Gamma'\cong\ZE_2$. The vertical arrows are the resolution maps
(essentially blow-ups) locally modelled on
$T^{3}\times U\to T^3\times(\CX^2/\pm 1)$, with $U$ a neighbourhood of the
exceptional divisor in an Eguchi-Hanson space. Note that there is a unique way
to lift the action of $\Psi$ to $M'$, so that the diagram~\eqref{res}
commutes. (One can further `fill in' the top left corner of~\eqref{res}, the
respective manifold being essentially the blow-up of the fixed point set of
$\Psi$ in $M'$, but we won't need that.)

The singular locus of $M'/\Psi$ consists of $4$ copies of $T^{3}$,
corresponding to the fixed point set of~$\gamma$, cf.~\S\ref{simplesub}.
We can choose the resolutions in constructing $\tv'_s$ so that it becomes
$\Psi$-invariant, moreover, so that away from a neighbourhood $S$ of the fixed
point set of $\Psi$, $\tv'_s$ is the pull-back of $\phin_s$ via
$M'{\setminus}S \to M$. 
Then $\varphi'_s$ is $\Psi$-invariant too, so both $\tv'_s$ and $\varphi'_s$
descend to well-defined \gtstr s on the quotient~$M'/\Psi$.
A neighbourhood of each $T^3$ component of the singular locus is homeomorphic
to $T^{3}\times(\CX/\{\pm 1\})$.
However, a consequence of our previous step is that the \gtstr\
$\varphi'_s$ on $M'/\Psi$ is not necessarily flat near the singular locus.
Therefore, we cannot immediately use Joyce's method discussed
in~\S\ref{simplesub}, resolving the singularities of $M'/\Psi$ by patching
$\varphi'_s$ with the product \gtstr\ on $T^{3} \times U$, in a way that keeps
the torsion small.

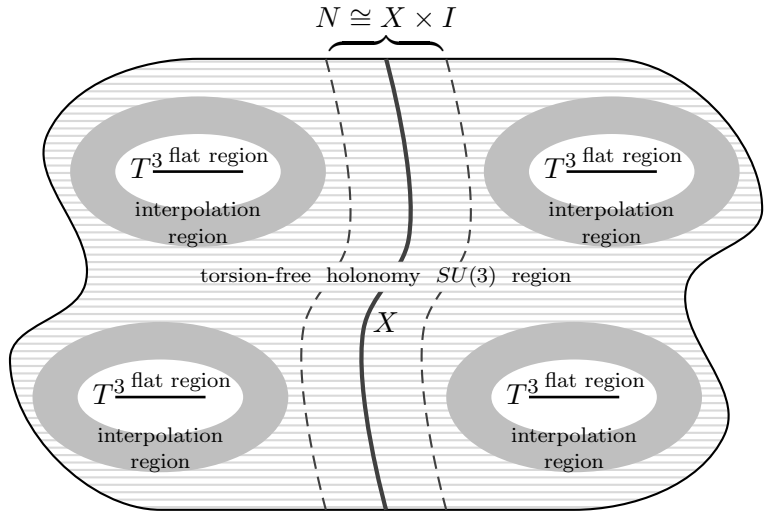
\begin{figure}[b]

\psset{unit=1mm}
\begin{pspicture}(140,65)
\psset{origin={70,30}}
\psset{linecolor=black}
\newgray{hatchgray}{0.85}
\newgray{vlightgray}{0.95}
\newgray{intergray}{0.75}
\psset{hatchcolor=hatchgray,fillcolor=vlightgray,hatchsep=2pt}
\pscustom[fillstyle=hlines,hatchangle=0]{
\psecurve(10,10)(30,30)(50,10)(40,0)(45,-20)(25,-30)(5,-20)
\psline(-30,-30)
\psecurve(-10,-10)(-30,-30)(-50,-10)(-40,0)(-45,20)(-25,30)(-5,20)
\psline(30,30)
}

\psclip{\pspolygon[linestyle=none]
(30,3)(-30,3)(-30,-1)(30,-1)(30,-100)(-30,0)(30,100)}
\pscurve[linecolor=darkgray,linewidth=1.6pt](0,30)(3,6)(-3,-6)(0,-30)
\multips(-8,0)(16,0){2}{
\pscurve[linestyle=dashed,linecolor=darkgray](0,30)(3,6)(-3,-6)(0,-30)}
\endpsclip

\multips(-30,-15)(5,30){2}
{\multips(0,0)(55,0){2}
{\psellipse*[linecolor=intergray](17,10)
\psellipse*[linecolor=white](11,5)
\psline[linewidth=1pt](-6,0)(6,0)}}

\rput(70,30){
\rput(0,-5){$X$}
\multirput(-30,-15)(5,30){2}{
\multirput(0,0)(55,0){2}{
\rput(2.75, 2) {\scriptsize flat region}
\rput(-6.75,0.75){$T^3$}
\rput(0,-7)
{\scriptsize\begin{tabular}{c}interpolation \\ region \end{tabular}}
}}
\rput(0,34){$\overbrace{\hspace{1.5cm}}^{\textstyle N \cong X \times I}$}
\rput(0,1){\scriptsize torsion-free \ holonomy \ $SU(3)$ \ region}}
\end{pspicture}

\caption{An interpolating \gtstr\ $\tv'_s + d(\eta' -\rho\chi)$
on the orbifold $M'/\Psi$.}
\label{thepic}
\end{figure}

On the other hand, the \gtstr{} $\tv'_s$ on $M'$ \emph{is} flat except near
the resolved singularities. In particular, $\tv'_s$ is flat near the fixed
point set $F \subset M'$ of $\Psi$, since the elements of
$\Gamma$ have disjoint fixed point sets.
We now wish to define on~$M'$, for $0 < s < \kappa'$, a closed
$\Psi$-invariant \gtstr{} with small torsion, by smoothly interpolating
between the flat $\tv_s'$ near~$F$ and the torsion-free $\varphi'_s$ in a
$\Psi$-invariant region $N' = \bigl( (\quart{-}\eps,\quart{+}\eps) \cup
(\tquart{-}\eps,\tquart{+}\eps)\bigr)\times X\subset (\RE/\ZE)\times X\simeq M'$,
for some $0 < \epsilon < \quart$. Note that although $N'$ has two components,
its image in the resolution $M$ of $M'/\Psi$ is connected and will be the
cylindrical neck region $N$ in the statement of Proposition
\ref{intermed.phi}. See Figure~\ref{thepic}.
To achieve small torsion, we use a generalization of the classical Poincaré
inequality.

\begin{lem}
\label{poincarelem}
Let $F$ be a compact Riemannian manifold and $I$ a bounded open interval.
For any $n \geq 0$, $k \geq 0$ and $p \geq 1$ there is a constant
$C_{n,p,k} > 0$, such that for every exact $\sobp{k}$ $m$-form $d\eta$ on the
Riemannian product $S = F \times I^{n}$ there is an $(m-1)$-form
$\chi$ with $d\chi = d\eta$ and
\begin{equation}
\label{poinesteq}
\norm{\chi}_{\sobp{k+1}} < C_{n,p,k} \norm{d\eta}_{\sobp{k}} .
\end{equation}
\end{lem}

\begin{proof}
The proof is by induction on $n$. The result holds for $n = 0$ by standard
Hodge theory and elliptic estimate for the Laplacian on compact~$F$.
For the inductive step, we show that if a manifold $S$ satisfies the assertion
of the lemma, then so does $S \times I$ with the product metric.

Let $t$ denote the coordinate on $I$ and $S_{t}$ denote the hypersurface
$S\times \{t\}$. We can write
\[ d\eta = \alpha + dt \wedge \beta , \]
with $\alpha$ and $\beta$ sections of the pull-back of $\Lambda^{*}T^{*}S$
to $S \times I$.
Write $\alpha(t)$, $\beta(t)$ for the corresponding forms on $S_{t}$.
Fix $t_{0} \in I$ and let
\[ \chi_{1}(t) = \int_{t_{0}}^{t} \beta(u) du . \]
Let $\nabla$ denote the covariant derivative on $S \times I$, and consider
$\chi_{1}$ as a form on $S \times I$. For any $0 \leq i \leq k$ and $t\in I$
\begin{multline*}
\norm{(\nabla^{i}\chi_{1})(t)}^{p}_{L^{p}(S_{t})} =
\int_{S} \left\Vert\int_{t_{0}}^{t} (\nabla^{i}\beta)(u) du \right\Vert ^{p}
vol_{S} \\ \leq
V^{p-1} \int_{S} \int_{t_{0}}^{t} \norm{(\nabla^{i}\beta)(u)}^{p} du \,vol_{S}
\leq
V^{p-1} \norm{\nabla^{i}\beta}^{p}_{L^{p}(S \times I)} ,
\end{multline*}
where $V$ is the length of $I$. Hence
\[ \norm{\nabla^{i}\chi_{1}}^{p}_{L^{p}(S \times I)} \leq
\int_{I} \norm{(\nabla^{i}\chi_{1})(u)}^{p}_{L^{p}(S_{u})} du
\leq V^{p} \norm{\nabla^{i}\beta}^{p}_{L^{p}(S \times I)} , \]
and
\[ \norm{\chi_{1}}_{\sobp{k}(S \times I)} \leq
V \norm{d\eta}_{\sobp{k}(S \times I)} . \]
$d(\eta - \chi_{1})$ has no $dt$-component, so the $dt$-component of
$d^{2}(\eta - \chi_{1})$ is $\contrat d(\eta - \chi_{1}) = 0$.
Hence $d(\eta - \chi_{1})$ is the pull-back to $S \times I$ of an exact
form on $S$. By the inductive hypothesis there is a form $\chi_{2}$ such that
$d\chi_{2} = d(\eta - \chi_{1})$ and $\chi = \chi_{1} + \chi_{2}$ satisfies
(\ref{poinesteq}) for some $C$ independent of~$d\eta$.
\end{proof}

Let $S \cong F \times I^4$ be a tubular neighbourhood of $F$ in $M'$. 
Applying Lemma~\ref{poincarelem} to $d\eta'_s$ in~\eqref{perturb}, we obtain
a $2$-form $\chi_s$ on $S$ such that
\[ d\chi_s = d\eta'_s|_{S} \]
and $\chi_s$ satisfies the $L^2$ estimate
\begin{subequations}\label{L2.L14}
\begin{equation}
\norm{\chi_s}_{L^2} < C_{4,2,0}\norm{d\eta'_s|_{S}}_{L^2} \le K_2s^4
\end{equation}
as well as the $\sobf{1}$ estimate
\begin{equation}
\norm{\chi_s}_{\sobf{1}} <
C_{4,14,0}\norm{d\eta'_s|_{S}}_{L^{14}} \le
C_{4,14,0}\vol(S)^{1/14}\norm{d\eta'_s|_{S}}_{C^{0}} <
K_{14} s^{1/2}
\end{equation}
\end{subequations}
with $K_2,K_{14}$ independent of~$s$. Here we also used~\eqref{tor}. We shall
also need an estimate on the uniform norm of $\chi_s$ which is obtained
from~\eqref{L2.L14} and the following version of Sobolev embedding.

\begin{thm}[{\cite[Theorem G1]{joyce00}}]
\label{g1thm}
Let $\mu$, $\nu$ and $s$ be positive constants, and suppose $M$ is a complete
Riemannian $7$-manifold, whose injectivity radius $\delta$ and
Riemannian curvature $R$ satisfy $\delta \geq \mu s$ and
$\norm{R}_{C^{0}} \leq \nu s^{-2}$. Then there exists $C > 0$ depending only
on $\mu$ and $\nu$, such that if
$\chi \in \sobf{1}(\Lambda^{3}) \cap L^{2}(\Lambda^{3})$ then
\begin{equation*}
\norm{\chi}_{C^{0}} \leq
C(s^{1/2}\norm{\nabla\chi}_{L^{14}} + s^{-7/2}\norm{\chi}_{L^{2}}) .
\end{equation*}
\end{thm}

We deduce that
$$
\norm{\chi_s}_{C^{0}} < C(K_2 s +  K_{14}s^{1/2}) < \widetilde{C}s^{1/2}
$$
as $s>0$ varies in a bounded interval.

Let $\rho$ be a cut-off function (not depending on $s$) which is $1$ near $F$
and $0$ outside $S$. Then
\begin{equation}
\label{drcsmalleq}
\norm{d(\rho\chi)}_{L^{2}} < K'' s^{4}, \;\:
\norm{d(\rho\chi)}_{C^{0}} < K'' s^{1/2}, \;\:
\norm{\nabla d(\rho\chi)}_{L^{14}} < K'',
\end{equation}
with $K''$ independent of $s$.
\begin{rmk}
\label{disjointrmk}
A key point in achieving the estimates \eqref{L2.L14} and \eqref{drcsmalleq}
is that a tubular neighbourhood $S\cong F \times I^4$ does not meet the region
affected by resolution of singularities in our first step. Therefore, the
metric on~$S$ and the respective constants in~\eqref{poinesteq} can be taken
to be independent of~$s$. See also Remark~\ref{generalisermk} below.
\end{rmk}
 For each $0 < s < \kappa'$,
$\tv'_s + d(\eta'_s -\rho\chi_s)$ is a closed
\gtstr{} which is flat near~$F$. It is clear from the chain rule that
it has small torsion in the sense of Theorem \ref{eacperturbthm}:
there is a form $\psi'_s$ such that $d{*}\psi'_s =
d\Theta(\tv'_s + d(\eta'_s -\rho\chi_s))$, satisfying \eqref{psiesteq}.
(Here $\Theta$~denotes the non-linear mapping $\varphi \mapsto *_\varphi
\varphi$; note that $\Theta$ depends only on the smooth structure and
orientation on~$M'$.)
However, we need to take care to choose $\psi'_s$ in such a way
that it vanishes not only on the cylindrical region $N'$, but also near~$F$.
Because $F$ has dimension $3$ any closed $4$-form on
the tubular neighbourhood $S$ is exact. By Lemma \ref{poincarelem} we can write
\begin{equation*}
(\Theta(\tv'_s + d\eta'_s) - \Theta(\tv'_s))|_{S} = d\chi'_s
\end{equation*}
for some $3$-form $\chi'_s$ on $S$, so that 
$d(\rho\chi'_s)$ satisfies estimates of the form (\ref{drcsmalleq}).
We can then take
\begin{equation*}
\psi'_s =
*(\Theta(\tv'_s + d(\eta'_s -\rho\chi_s))- \Theta(\tv'_s + d\eta'_s) + 
d(\rho\chi'_s)) .
\end{equation*}
This is supported in $S$ and vanishes near $F$ and satisfies \eqref{psiesteq}
for some $\lambda > 0$ (depending on the constants $K'$ and $K''$ from
\eqref{tor} and \eqref{drcsmalleq}, but \emph{not} on $s$).
We can ensure that all forms are $\Psi$-invariant, so
$\psi'_s$ descends to a small $3$-form, still denoted by $\psi'_s$ on the
orbifold $M'/\Psi$. As this form is supported away from the singular locus,
$\psi'_s$ is also well-defined on the resolution $M$.

For $0 < s < \kappa'$, the form $\tv'_s + d(\eta'_s -\rho\chi_s)$ descends to an
orbifold \gtstr{} on
$M'/\Psi$ with small torsion. By construction, it is a product \gtstr{}
on the image $N \cong I \times X \subset M'/\Psi$ of $N' \subset M'$.
Its orbifold singularities are modelled on quotients
of the flat \gtstr{}, so the singularities can be resolved
like in \S \ref{simplesub}
to define a closed \gtstr{} $\tv_s$ on~$M$.
We make sure that the Eguchi--Hanson spaces used in this resolution have
the same scale as those used for the resolution of the first-step singularities.
The torsion introduced by the resolution is then small, in the sense
that there is a smooth $3$-form $\psi''_s$ on $M$,
supported near the pre-image $F'$ of the singular locus, such that
$d^{*}\psi''_s = d^{*}\tv_s$ near $F'$ and $\psi''_s$ satisfies the estimates
\eqref{psiesteq}.
Thus for each $0 < s < \kappa'$, $\tv_s$ is a \gtstr{} on $M$ with small
torsion (controlled by $\psi_s = \psi'_s + \psi''_s$) and $N$ is a cylindrical
neck region, so that $\tv_s$ satisfies the claims (i)-(iii) of
Proposition~\ref{intermed.phi}.

To prove the remaining claim~\ref{exactitem} we identify the difference
between our $\tv_s$ and the \gtstr{} $\phin_s$ obtained (in \S\ref{simplesub})
by resolving all the singularities of $T^7/\Gamma$ in a single step.
By construction in the previous paragraph, $\tv_s-\phin_s$ vanishes on a
neighbourhood of the pre-image in $M$ of the singular locus of $M'/\Psi$
(see~\eqref{res}). Therefore, we may interchangeably consider
$\tv_s-\phin_s$ as a $\Psi$-invariant form on~$M'$ supported away from a
neighbourhood $S$ of the fixed point set of~$\Psi$.

Now recall that $\tv'_s$ is a $\Psi$-invariant form on~$M'$ and
the restriction of $\tv'_s$ agrees with the pull-back of~$\phin_s$ to
$M{\setminus}S$. On the other hand, the difference between the pull-back
of~$\tv_s$ to $M{\setminus}S$ and $\tv'_s|_{M'{\setminus}S}$ is
$d(\eta'_s -\rho\chi_s)$. The 2-form $\eta'_s -\rho\chi_s$ is
$\Psi$-invariant, as $\eta'_s$ and $\chi_s$ are so. As
$\eta'_s -\rho\chi_s$ is also supported away from~$S$, it is the pull-back
via $M'{\setminus}S \to M$ of a well-defined 2-form, say $\xi$, on~$M$.
We find that $\tv_s-\phin_s$ is the exact form $d\xi$.
This completes the proof of Proposition~\ref{intermed.phi}.

\section{Further examples and applications}
\label{furthersec}

We now construct a few further examples of EAC \gtmfd s with different types
of cross-sections and discuss their topology. We also give examples of
EAC coassociative submanifolds.

\subsection{Topology of the example of \S \ref{exsec}}
\label{five.one}
To study the topology of the EAC \gtmfd{} $M_{+}$ we consider it as
a resolution of $(T^{6} \times \bbr)/\Gamma$.
As noted in \S\ref{EACex}, both the orbifold and its resolution 
are simply-connected.

Recall that we chose $\Gamma'$ to be the stabiliser of the $S^{1}$ factor
corresponding to the $x_{1}$ coordinate in \eqref{simplegammaeq},
i.e. $\Gamma' = \gen{\alpha,\beta}$.
The resolution of the intermediate quotient $S^{1} \times T^{6}/\Gamma'$ is
isomorphic to $S^{1} \times X_{19}$, for a simply-connected Calabi--Yau $3$-fold
$X_{19}$. This $X_{19}$ is then the cross-section of $M_{+}$.

We find that the Betti numbers of $(T^{6} \times \bbr)/\Gamma$ are $b^{2}= 0$,
$b^{3}= 4$, $b^{4}= 3$, $b^{5}= 0$.
The singular locus in $(\bbr \times T^{6})/\Gamma$ consists of
$8$ copies of $T^{2} \times \bbr$ and $2$ copies of $T^{3}$.
Resolving the former adds $1$, $2$ and $1$ to $b^{2}$,
$b^{3}$ and $b^{4}$, respectively.
Therefore
\begin{equation*}
\begin{aligned}
b^{2}(M_{+}) &= 8 \cdot 1 + 2 \cdot 1 = 10 , \\
b^{3}(M_{+}) &= 4 + 8 \cdot 2 + 2 \cdot 3 = 26 , \\
b^{4}(M_{+}) &= 3 + 8 \cdot 1 + 2 \cdot 3 = 17 , \\
b^{5}(M_{+}) &= 2 \cdot 1 = 2 .
\end{aligned}
\end{equation*}
We can also compute the Betti numbers of the cross-section $X_{19}$, and
find that $b^{2}(X_{19}) = 19$, $b^{3}(X_{19}) = 40$. Therefore its Hodge
numbers are
\[ h^{1,1}(X_{19}) = h^{1,2}(X_{19}) = 19 . \]
\begin{rmk}
The Calabi--Yau 3-fold $X_{19}$ can be obtained in a slightly different
way. Blowing up the singularities of $T^6/\gen{\alpha}$ gives a product
of a Kummer K3 surface and an elliptic curve $\cale \cong T^2$. The map $\beta$
descends to a holomorphic involution of $\mathrm{K3}\times\cale$, still denoted
by~$\beta$. The restriction $\beta|_{\cale}$ induced by $-1$ on $\CX$ has 4
fixed points in~$\cale$ and $(\beta|_{\mathrm{K3}})^*$ multiplies the
holomorphic (2,0)-forms on the K3 surface by $-1$. The 3-fold $X_{19}$ is then
the blow-up of the orbifold $(K3 \times \cale)/\gen{\beta}$ at its singular
locus. Calabi--Yau 3-folds obtained from $\mathrm{K3}\times \cale$ and an
involution $\beta$ with the above properties were studied by Borcea
\cite{borcea97} and Voisin~\cite{voisin93} in connection with mirror symmetry,
and are sometimes called \emph{Borcea--Voisin manifolds}.
\end{rmk}

According to \cite[Propn.~3.5]{jn1}, the dimension of the moduli
space of torsion-free EAC \gtstr s on $M_{\gi}$ can be written in terms of Betti
numbers as
\begin{equation}
\label{dimmodeq}
b^{4}(M_{\gi}) + \half b^{3}(X) - b^{1}(M_{\gi}) - 1 ,
\end{equation}
so in this example we find that the moduli space has dimension $36$.

\subsection{Two more EAC \gtmfd s}
\label{more.examples}

Let us consider some variations of the example in the previous
subsection in order to get examples of different topological types.
Especially, we want to show that an EAC manifold with holonomy exactly $G_{2}$
may have a cross-section $X$ whose holonomy is a proper subgroup of $SU(3)$.
Here and below by holonomy of a cross-section we mean `holonomy at
infinity', corresponding to the Calabi--Yau structure on $X$ defined by the
asymptotic limit of \gtstr\ along the cylindrical end (cf.~\S\ref{as.cyl}).

When we let the group $\Gamma$ from \eqref{simplegammaeq} act on
$\bbr \times T^{6}$ in the previous subsection, we could have taken the
$\bbr$-factor to correspond to a coordinate on $T^{7}$ other than $x_{1}$.
In the geometric interpretation of Remark \ref{cutrmk} this means pulling
the compact \gtmfd{} $M$ apart along a hypersurface defined by $x_{i}=\const$
rather than $x_{1} = \const$.
Pulling apart $M$ in the $x_{2}$ or $x_{4}$ direction
we get essentially the same pair of 7-manifolds $M_\pm$ as for the $x_{1}$
direction in \S\ref{EACex}. We just need to use
$\gen{\gamma, \alpha}$ or $\gen{\beta, \gamma}$ as $\Gamma'$ to define
the intermediate resolution.

\label{ex2item}
If we pull apart along the $x_{3}$ direction we get a slightly
different geometry and new examples. The subgroup of $\Gamma$ acting trivially
on the $x_{3}$ factor is $\Gamma' = \gen{\alpha, \beta\gamma}$,
which only contains one non-identity element with fixed points.
The cross-section of the neck is a resolution $X_{11}$ of
$T^{6}/\Gamma'$. It is a non-singular quotient of $T^{2} \times K3$ by an
involution that acts as $-1$ on the $T^2$ factor, so the first Betti number
$b^{1}(X_{11})$ vanishes,
but the holonomy of $X_{11}$ is $\bbz_{2} \ltimes SU(2)$.
The EAC \gtmfd s $M_{\gi}$ are however simply-connected with
a single cylindrical end.
Thus, by Corollary~\ref{irred}, these are examples of irreducible EAC
\gtmfd s with locally reducible cross-section. These are not
homeomorphic to the example in \S\ref{EACex} as the cross-section $X_{19}$ of
the latter example is simply-connected, whereas $X_{11}$ is not. 
We can also compute the Betti numbers of~$M_\pm$.

In the present case, the singular locus in each half is $4$ copies of $T^{3}$
and $4$ copies of $T^{2} \times \bbr$. The Betti numbers are therefore
\begin{equation*}
\begin{aligned}
b^{2}(M_{\gi}) &= 4 \cdot 1 + 4 \cdot 1 = 8 , \\
b^{3}(M_{\gi}) &= 4 + 4 \cdot 2 + 4 \cdot 3 = 24 , \\
b^{4}(M_{\gi}) &= 3 + 4 \cdot 1 + 4 \cdot 3 = 19 , \\
b^{5}(M_{\gi}) &= 4 \cdot 1 = 4 .
\end{aligned}
\end{equation*}
The Hodge numbers of $X_{11} = (T^{2} \times K3)/\bbz_{2}$ are
\[ h^{1,1}(X_{11}) = h^{1,2}(X_{11}) = 11 . \]
By the formula \eqref{dimmodeq} the moduli space of torsion-free EAC \gtstr s on
$M_{\gi}$ has dimension~$31$.

It is also possible to pull apart $M$ in the $x_5$, $x_6$ or $x_7$ directions.
In all three cases the resulting EAC \gtmfd s have $b^1(M_{\gi}) = 1$,
so do not have full holonomy $G_2$. In \S\ref{connectsums} we shall consider
the case of $x_5$ in greater detail, and relate $M_\pm$ to quasiprojective
complex \mbox{3-folds} with holonomy $SU(3)$ and to a `connected-sum
construction' of compact irreducible 
\gtmfd s \cite{kovalev03,kovalev-lee08}. The case $x_6$ is similar, but
the case $x_7$ is qualitatively different in that the cross-section $X$ is
not $T^2 \times K3$ but a non-singular quotient of $T^6$.

In order to find an example of an EAC manifold with holonomy $G_2$ whose
cross-section at infinity is flat, we replace $\Gamma$ with the group
$\Gamma_{1}$
generated by
\begin{equation}
\label{altgammeq}
\begin{aligned}
\alpha &: (x_{1}, \ldots, x_{7}) \mapsto
(\phantom{-}x_{1}, \phantom{-}x_{2}, \phantom{\half{-}}x_{3}, -x_{4}, 
\phantom{\half}{-}x_{5}, \phantom{\half}{-}x_{6}, -x_{7}) , \\
\beta &: (x_{1}, \ldots, x_{7}) \mapsto
(\phantom{-}x_{1}, -x_{2}, \phantom{\half}{-}x_{3}, \phantom{-} x_{4},
\phantom{\half{-}} x_{5}, \half {-} x_{6}, -x_{7}) , \\
\gamma_{1} &: (x_{1}, \ldots, x_{7}) \mapsto
(-x_{1}, \phantom{-}x_{2}, \half{-}x_{3}, \phantom{-} x_{4},
\half{-}x_{5}, \phantom{\half{-}} x_{6}, -x_{7}) .
\end{aligned}
\end{equation}
The orbifold $T^{7}/\Gamma_{1}$ can be resolved in the same way as
$T^{7}/\Gamma$, and the resulting compact \gtmfd{} $M_{1}$ has the same Betti
numbers as $M$. 
Pulling $M_{1}$ apart in the $x_{7}$ direction gives an EAC manifold with
holonomy exactly $G_{2}$ whose cross-section is the non-singular quotient of
$T^{6}$ by $\Gamma' = \gen{\alpha\beta, \beta\gamma_1} \cong \bbz_{2}^{2}$.
In particular, the cross-section is flat (in this
case there is no need for any intermediate resolution in the construction
of the EAC \gtstr).
The manifold has Betti numbers
\begin{equation*}
\begin{aligned}
b^{2}(M_{+}) &=  6 \cdot 1 = 6 , \\
b^{3}(M_{+}) &= 4 + 6 \cdot 3 = 22 , \\
b^{4}(M_{+}) &= 3 + 6 \cdot 3 = 20 , \\
b^{5}(M_{+}) &= 6 \cdot 1 = 6 .
\end{aligned}
\end{equation*}
The cross-section has $b^1(T^{6}/\bbz_{2}^{2}) = 0$ (this is in any case
a necessary condition for the EAC manifolds $M_{\gi}$  to have full holonomy
$G_2$, by \cite[Propn.~5.16]{jn1} and Theorem~\ref{irredthm}) and 
\[ h^{1,1}(T^{6}/\bbz_{2}^{2}) = h^{1,2}(T^{6}/\bbz_{2}^{2})  = 3 . \]
The moduli space of torsion-free EAC \gtstr s on $M_{\gi}$ has dimension $23$.

\begin{rmk}
\label{generalisermk}
Looking carefully, the argument for pulling apart a compact \gtmfd{} obtained
by resolving $T^7/\Gamma$ (provided a method for resolving its
singularities with small torsion)
relies on two properties of the group $\Gamma$.
The first is that $\Gamma$ preserves a product decomposition
$T^7 = S^1 \times T^6$, with some elements acting as reflections on the $S^1$
factor. The other is that, in order to apply Lemma \ref{poincarelem}, the
fixed point sets of elements of the subgroup $\Gamma'$ acting trivially on the
$S^1$ factor must not intersect fixed point sets of the remaining elements
(cf. Remark~\ref{disjointrmk}).

In \cite{joyce00}, Joyce gives a number of examples of suitable groups
$\Gamma$, where such fixed point sets of elements are pair-wise disjoint.
Most of them preserve a product decomposition, so can be pulled apart
(possibly in more than one way) giving further examples of EAC
$G_2$-manifolds.

More generally, a method is proposed in~\cite[p. 304]{joyce00} for
constructing \gtstr s with small torsion on a resolution of singularities of
$S^1\times X^6/(-1,a)$, where $X^6$ is a Calabi--Yau $3$-fold and $a$ is
an anti-holomorphic involution on~$X^6$. As discussed in \S\ref{su2sub}, the
Calabi--Yau structure of $X$ is completely determined by two closed forms, the
real part $\Omega$ of a non-vanishing holomorphic $(3,0)$-form and the Kähler
form~$\omega$. Then $a^{*}\omega = -\omega$ and without loss of generality
$a^{*}\Omega = \Omega$. The product torsion-free \gtstr\
$\Omega + dt\wedge\omega$ as in~\eqref{g2su3} is well-defined on
$S^{1} \times X$ and invariant under $(-1, a)$, thus descends to a
well-defined \gtstr{} on the quotient. The singular locus of
$S^1\times X^6/(-1,a)$ is of the form $\{0, \half \} \times L$, where
$L\subset X$ is the fixed point set of~$a$, necessarily a real
$3$-dimensional submanifold of~$X$ (more precisely, $L$ is special
Lagrangian).

A resolution of singularities of $(S^{1} \times X)/(-1, a)$ should be locally
modelled on $\bbr^{3} \times Y$, where $Y$ is an Eguchi-Hanson space.
It is explained in~\cite[p. 304]{joyce00} that to get a well-defined \gtstr\
(initially with small torsion) on the resolution one would need to make a
choice of smooth family of ALE hyper-Kähler metrics on~$Y$.

Assuming such choice, one could equally well define EAC \gtstr s with small
torsion on $(\bbr \times X)/(-1, a)$, and use Theorem \ref{eacperturbthm} to
obtain EAC manifolds with holonomy $G_2$.
\end{rmk}

\subsection{EAC coassociative submanifolds}
\label{coasssub}

Let $M$ be a 7-manifold with a \gtstr{} given by a 3-form $\varphi$.
A \emph{coassociative submanifold} $C\subset M$ is a 4-dimensional submanifold
such that $\varphi|_C=0$. It is not difficult to check that then the 4-form
$*_\varphi \varphi$ never vanishes on~$C$, thus every coassociative submanifold
is necessarily orientable.

If a \gtstr{} $\varphi$ is torsion-free then $d{*_\varphi}\varphi=0$ and the
4-form $*_\varphi \varphi$ is a \emph{calibration} on~$M$ as defined by Harvey
and Lawson \cite{harvey82}. In this case, coassociative submanifolds
(considered with appropriate orientation) are precisely the submanifolds
calibrated by $*_\varphi \varphi$, in particular, every coassociative
submanifold of a \gtmfd{} is a minimal
submanifold~\cite[Theorem II.4.2]{harvey82}. Our definition of coassociative
submanifold is not the same as in {\it op.cit.} but is equivalent to it via
\cite[Propn.~IV.4.5 \&\ Theorem~IV.4.6]{harvey82}.

One way of producing examples of coassociative submanifolds is provided by the
following.

\begin{prop}[{\cite[Propn.~10.8.5]{joyce00}}]
\label{coassprop}
Let $\sigma : M \to M$ be an involution such that $\sigma^{*}\varphi =
-\varphi$. Then each connected component of the fixed point set of $\sigma$ is
either a coassociative $4$-fold or a single point.
\end{prop}

Any $\sigma$ as in the hypothesis of Proposition~\ref{coassprop} is called an
\emph{anti-$G_2$ involution}. It is necessarily an isometry of~$M$.

Let $M^{7}$ be the compact \gtmfd{} discussed in \S \ref{simplesub}
and $\varphi$ its torsion-free \gtstr. We shall consider two examples of
anti-$G_2$ involution taken from \cite[\S 12.6]{joyce00} which extend to
well-defined anti-$G_2$ involutions of EAC \gtmfd s constructed in
\S \ref{more.examples}.

\begin{ex}
\label{coass1ex}
Define an orientation-reversing isometry of $T^{7}$
as in \cite[Example 12.6.4]{joyce00}.
\begin{equation}
\label{sigmaeq}
\sigma : (x_{1}, \ldots, x_{7}) \mapsto
(\half{-}x_{1}, x_{2}, x_{3}, x_{4}, x_{5}, \half{-}x_{6}, \half{-}x_{7}) .
\end{equation}
Then $\sigma$ commutes with the action of $\Gamma$ defined
by \eqref{simplegammaeq}
and pulls back $\varphi_{0}$ to $-\varphi_{0}$. When the singularities
of $T^{7}/\Gamma$ are resolved to form the compact \gtmfd{} $M$ one can ensure
that $\sigma$ lifts to an anti-$G_2$ involution of $(M,\varphi)$. The fixed
point set of $\sigma$ in $M$ consists of 16 isolated points and one copy of
$T^{4}$, which is a coassociative submanifold of $M$.

We can also consider $\sigma$ in (\ref{sigmaeq}) as an involution of
$T^{6} \times \bbr$. Provided that the $\bbr$ factor corresponds to the
$x_{2}$, $x_{3}$ or $x_{4}$ coordinate this again commutes with the
action of~$\Gamma$. When we pull apart $M$ in
the $x_{2}$, $x_{3}$ or $x_{4}$ direction the resulting irreducible
EAC \gtmfd s $M_{\pm}$
are resolutions of $(T^{6} \times \bbr)/\Gamma$, so $\sigma$ lifts to
an anti-$G_2$ involution of~$M_{\gi}$.
The fixed point set in each half $M_{\gi}$ consists of $8$ isolated points
and one 4-manifold $C_{\gi} \cong T^{3} \times \bbr$, which is an
asymptotically cylindrical coassociative submanifold of $M_{\gi}$ (in the
obvious coordinates for the cylindrical end of $M_{\gi}$, $C_{\gi}$ is a
product submanifold).
\end{ex}

\begin{ex}
\label{coass2ex}
Here is another orientation-reversing isometry of $T^{7}$ taken from
\cite[Example 12.6.4]{joyce00}.
\begin{equation*}
\sigma : (x_{1}, \ldots, x_{7}) \mapsto
(\half{-}x_{1}, \half{-}x_{2}, \half{-}x_{3}, x_{4}, x_{5}, x_{6}, x_{7}) .
\end{equation*}
Its fixed point set in $T^{7}/\Gamma$ consists of 16 isolated points and two
copies of $T^{4}/\{\pm1\}$. Again, $\sigma$ lifts to an anti-$G_2$ involution
of $(M,\varphi)$ and the corresponding coassociative submanifolds in $M$
are now, respectively, two copies of the usual Kummer resolution of
$T^{4}/\{\pm1\}$, diffeomorphic to a $K3$ surface.

If we pull apart $M$ in the $x_{4}$ direction then $\sigma$ again defines
anti-$G_2$ involutions of the resulting irreducible EAC \gtmfd s $M_{\pm}$.
In each half the fixed point set has two $4$-dimensional components, which
are resolutions of $(T^{3} \times \bbr)/\{\pm1\}$. These are asymptotically
cylindrical coassociative submanifolds of $M$. Topologically, they are `halves'
of a $K3$ surface: attaching two copies by identifying their boundaries $T^3$
`at infinity' via an orientation-reversing diffeomorphism one obtains a closed
4-manifold diffeomorphic to~$K3$.
\end{ex}

Compact coassociative submanifolds have a well-behaved deformation theory.
For any coassociative submanifold $C \subset M$ the normal bundle of $C$ is
isomorphic to the bundle $\Lambda^{2}_{+}T^{*}C$ of self-dual 2-forms.
McLean \cite[Theorem 4.5]{mclean98} shows that the nearby coassociative
deformations of a closed coassociative submanifold~$C$ is a smooth manifold of
dimension $b^{2}_{+}(C)$ (see also \cite[Theorem 2.5]{joyce05}).

Joyce and Salur prove an EAC analogue of McLean's result. Denote by
$H^{2}_{0}(C,\bbr) \subseteq H^{2}(C,\bbr)$ the subspace of cohomology
classes represented by compactly supported 2-forms. Equivalently,
$H^{2}_{0}(C,\bbr)$ is the image of the natural `inclusion homomorphism' of
the cohomology with with compact support
$H^{2}_{\mathrm{c}}(M,\RE)\to H^{2}(M,\RE)$.
\begin{prop}[{\cite{joyce05}}]
Let $M^{7}$ be an EAC \gtmfd{} with cross-section $X^{6}$ and
$C \subset M$ an EAC coassociative submanifold asymptotic to $\bbrp\times L$,
for a 3-dimensional submanifold $L \subset X$.
Then the space of nearby coassociative deformations of~$C$ asymptotic to
$\bbrp\times L$ is a smooth manifold of finite dimension $b^{2}_{0,+}(C)$,
which is the dimension of a maximal positive subspace for the intersection
form on $H^{2}_{0}(C,\bbr)$.
\end{prop}
For $T^{3} \times \bbr$ or the half-$K3$-surface this quantity vanishes.
Indeed, $H^i_0(T^3 \times \bbr) = 0$ for all~$i$. The half-K3-surface can
be regarded as a quotient of $T^3 \times \bbr$ blown-up at some
$\bbc^2/\{\pm1\}$ singularities, so the only contribution to $H^2_0$ comes from
the exceptional $\bbc P^1$ divisors, which have negative self-intersection.
Thus the coassociative submanifolds in example \ref{coass1ex} and
\ref{coass2ex} are rigid if their `boundary $L$ at infinity' is kept fixed.

\section{Pulling apart $G_{2}$-manifolds}
\label{pullsec}

In \S\ref{exsec} and \S\ref{furthersec} we constructed pairs of asymptotically
cylindrical \gtmfd s $(M_{\gi},\varphi_{s,\pm})$. They were obtained from a
decomposition \eqref{decomp} of compact \gtmfd s $(M,\varphi_s)$ taken
from~\cite{joyce00} which are resolutions of $T^7/\Gamma$.
In this section we show how our construction of $(M_{\gi},\varphi_{s,\pm})$
can be regarded as an inverse operation to a gluing construction
in~\cite{kovalev03} that forms compact \gtmfd s from
a `matching' pair of EAC \gtmfd s. It is easy to see that joining the
manifolds $M_\gi$ at their cylindrical ends yields a manifold diffeomorphic
to~$M$, but we shall prove a stronger statement that there is a continuous
path of torsion-free \gtstr s connecting $\varphi_s$ to the glued \gtstr s.
In other words, pulling the compact \gtmfd{} $(M,\varphi_s)$ apart into EAC
halves and gluing them back together again gives a \gtstr{} that is
deformation-equivalent to the original~$\varphi_s$.

We begin by describing the gluing construction of compact \gtmfd s from
a matching pair of EAC \gtmfd s.
Let $(M_\gi,\varphi_\gi)$ be some EAC \gtmfd s with cross-sections $X_\gi$.
The restrictions of the EAC torsion-free \gtstr s $\varphi_{\pm}$
to the cylindrical ends $[0,\infty)\times X_{\pm}\subset M_{\pm}$
have the asymptotic form
$$
\varphi_{\pm}|_{[0,\infty)\times X_{\pm}}=\varphi_{\pm,cyl}+d\eta_{\pm},
$$
where each
$$
\varphi_{\pm,cyl}=\Omega_{\pm} + dt \wedge \omega_{\pm}
$$
is a product cylindrical \gtstr{} induced by a Calabi--Yau structure on~$X$
and each 2-form $\eta_{\pm}$ decays with all derivatives at an exponential
rate as $t\to\infty$
$$
\|\nabla^r\eta_{\pm}\|_{\{t\}\times X_{\pm}}<C_re^{\lambda t}.
$$
We say that $\varphi_\gi$ is a matching pair of EAC \gtstr s if there is an
orientation-reversing diffeomorphism $F:X_+\to X_-$ satisfying
\begin{equation}
\label{matcheq}
F^*(\Omega_-)=\Omega_+,\qquad
F^*(\omega_-)=-\omega_+.
\end{equation}
For each sufficiently large $L>0$, the 3-form
$$
\tv_{\pm}(L)=\varphi_{\pm} - d(\alpha(t-L)\eta_{\pm})
$$
induces a well-defined \gtstr{}. Here we used $\alpha(t)$ to denote a
smooth cut-off function, $0\le\alpha(t)\le 1$, $\alpha(t)=0$ for
$t\le 0$ and $\alpha(t)=1$ for $t\ge 1$. 
For $L>1$, denote
$M_{\pm}(L)=M_{\pm}\setminus \bigl((L+1,\infty)\times X_\pm\bigr)$.
A {\em generalized connected sum} of $M_{\pm}$ may be defined as
$$
M(L)=M_{+}(L)\cup_F M_-(L)
$$
identifying the collar neighbourhoods of the boundaries of $M(L)$ via
$(t,x)\in [L,L+1]\times X_+\to (2L+1-t,F(x))\in [L,L+1]\times X_-$.
The 3-forms $\tv_{\pm}(L)$ agree on the `gluing region'
$[L,L+1]\times X_\pm$ and together define a closed $G_2$ 3-form $\varphi(L)$
on~$M(L)$. It is not difficult to check that the co-differential of this form,
relative to the metric $g(\varphi(L))$ satisfies
$$
\|d {*_{\varphi(L)}} \varphi(L)\|_{L^p_k(M(L))}<C_{p,k}e^{\lambda L},
$$
but need not vanish as the derivatives of the cut-off function introduce
`error terms'. Thus the \gtstr\ $\varphi(L)$ has `small' torsion on~$M$,
but need not be torsion-free.

For each $L$, the $M(L)$ is diffeomorphic, as a smooth manifold, to a fixed
compact 7-manifold $M$, but the metrics
$g(\varphi(L))$ have diameter asymptotic to $2L$, as $L\to\infty$.

\begin{thm}[{\cite[\S 5]{kovalev03}}]\label{gluing}
Let a compact 7-manifold $M(L)$ and a $G_2$ 3-form
$\varphi(L)\in\Omega^3_+(M(L))$ be a generalized connected sum of a pair of
EAC \gtmfd s $(M_\pm,\varphi_\pm)$ with \gtstr s satisfying~\eqref{matcheq}.

Then there exists an $L_0>1$ and for each $L>L_0$ a 2-form $\eta_L$ on~$M$,
so that the \gtstr{} on~$M$ induced by 
$\varphi(L)+d\eta_L$ is torsion-free.
Furthermore, the form $\eta_L$ may be chosen to satisfy
$\|\eta_L\|_{L^p_k(M(L))}<C_{p,k}e^{-\delta L}$, for some positive constants
$C_{p,k},\delta$ independent of~$L$.
\end{thm}
The above is a variant of the `gluing theorem' for solutions of non-linear
elliptic PDEs on generalized connected sums~\cite{KS}, adapted
to~\eqref{suff2eq}. The proof uses a lower bound for the linerization
of~\eqref{suff2eq} on~$M$ with carefully chosen weighted Sobolev norms
and an application of the inverse mapping theorem in Banach spaces.
\begin{defn}
For a matching pair of torsion-free \gtstr s and $L > L_0$, let
\[ \Phi(\varphi_{+}, \varphi_{-}, L) = \varphi(L) + d\eta_L \]
be the \gtstr\ on~$M$ defined in Theorem~\ref{gluing}.
\end{defn}

The family of $G_2$-metrics induced by $\Phi(\varphi_{+},\varphi_{-},L)$
may be thought of as 
stretching the neck of a generalized connected sum, defined by the
decomposition of compact 7-manifold $M$ along a hypersurface~$X$.
The pair of EAC \gtmfd s $(M_\pm,\varphi_\pm)$ may be identified as a boundary
point of the moduli space for \gtstr s on~$M$ corresponding to the limit of
the path 
$\Phi(\varphi_{+}, \varphi_{-}, L)$, as $L\to\infty$
(see~\cite[\S 5]{jn3} for more precise details).

Now we return to consider the pairs $(M_\gi,\varphi_\gi)$ of EAC
\gtmfd s constructed in \S\ref{EACex} and~\S\ref{furthersec}.
It follows from the decomposition \eqref{decomp} that $\varphi_\gi$ is
a matching pair of EAC \gtstr s in the sense of \eqref{matcheq}.
The generalized connected sum of $M_\gi$ is clearly diffeomorphic to~$M$
in the left-hand side of~(\ref{decomp}a), so
by Theorem~\ref{gluing} we obtain a family of torsion-free \gtstr s
$\Phi(\varphi_{+}, \varphi_{-}, L)\in\Omega^3_+(M)$. On the other hand,
in this case we can construct on $M$ another path $\phi(L)$ of
torsion-free \gtstr s,
with the same asymptotic properties as $L\to\infty$, using the
\gtstr\ $\tv_s$ defined in Proposition~\ref{intermed.phi}.
Recall from~\eqref{neckeq} that $\tv_s$ restricts to a product torsion-free
\gtstr\ on $N\subset M$, which is a finite cylindrical domain $N \cong
(-\epsilon,\epsilon)\times X$. For each $L\geq 0$, using a diffeomorphism
$f_L$ between intervals in~$\RE$
\begin{equation}\label{change.t}
t\in (-\epsilon,\epsilon)\to t_L=f_L(t)\in (-\epsilon-L,\epsilon+L),
\end{equation}
we define a new \gtstr\ $\tv_s(L)$ on~$M$ so that
$\tv_s(L)|_{N} = \Omega + dt_L \wedge \omega$ and $\tv_s(L)$ coincides with
$\tv_s$ away from $N$. It is easy to see that the resulting family of metrics
$g(\tv_s(L))$ may be informally described as `stretching' the neck region
$N$ in the Riemannian manifold $(M,g(\tv_s))$. The change~\eqref{change.t} of
the cylindrical coordinate on~$N$ amounts to the diameter of $(M,g(\tv_s))$
being increased by~$2L$.

For each $L\geq 0$, the \gtstr\ $\tv_s(L)$ satisfies the same estimates on
the torsion as $\tv_s$ (this follows from the argument of~\S\ref{four.three}).
Therefore, the same method as in the case of $\tv_s$ 
applies to show that $\tv_s(L)$ can be perturbed to a torsion-free 
\gtstr{} $\phi_s(L)=\tv_s(L)+(\text{exact form})$ 
\cite[\S 11.6 and \S 12.2]{joyce00}.

There is no obvious reason for the \gtstr s $\phi(L)$ to be
isomorphic to $\Phi(\varphi_{+}, \varphi_{-}, L)$, but we
show that the two families are `asymptotic' to each other in the following
sense.

\begin{thm}
\label{pullthm}
Let $M^{7}$ be a compact manifold with a closed \gtstr{} $\tv_s$,
such that the assertions \ref{neckitem}-\ref{psineckitem} of
Proposition~\ref{intermed.phi} hold, for each sufficiently small~$s$.
Assume that $s$ is sufficiently small and define the path $\phi_s(L)$
as above.
Let $\tv_{s,\pm}$ be EAC \gtstr s on the manifolds $M_\pm$ with cylindrical
ends defined after Proposition~\ref{intermed.phi} and $\varphi_{s,\gi}$
the torsion-free perturbations of $\tv_{s,\pm}$ within their cohomology
class defined by Theorem~\ref{eacperturbthm}.

Then for every sufficiently large $L$, there are some matching
deformations $\varphi'_{s,\gi}=\varphi'_{s,\gi}(L)$ of $\varphi_{s,\gi}$,
satisfying $\|\varphi'_{s,\gi}-\varphi_{s,\gi}\| < C_1 L^{-1}$, and a real
$\eps_L$, satisfying $|\eps_L|<C_2$, with $C_1,C_2>0$ independent of~$L$,
so that the \gtstr\ $\phi_s(L)$ is isomorphic to $\Phi(\varphi'_{s,+},
\varphi'_{s,-}, L+\eps_L)$.
\end{thm}

When $(M,\varphi_s)$ is an example of compact \gtmfd{} discussed
in \S \ref{simplesub} and \S\ref{furthersec}
we shall deduce from the proof of Theorem~\ref{pullthm} a further
result which will be used in \S\ref{connectsums}.

\begin{thm}\label{pullcor}
Let $(M,\varphi_s)$ and $(M_\gi,\varphi_{s,\gi})$ be the \gtmfd s defined in
\S\ref{simplesub} and \S\ref{EACex} or in \S\ref{furthersec}. There is, for
every sufficiently small $s>0$, a continuous path of torsion-free \gtstr s on
$M$ connecting $\varphi_s$ and $\Phi(\varphi_{s,+}, \varphi_{s,-}, L)$, whenever
$L$ is sufficiently large in the sense of Theorem~\ref{gluing}.
\end{thm}

As we shall see, a closed \gtstr{} $\tv_s$ will be required once again in the
argument of Theorem~\ref{pullcor} and the clause~\ref{exactitem} of
Proposition \ref{intermed.phi} will be important.

In order to prove Theorems \ref{pullthm} and~\ref{pullcor} we need
to recall some results concerning the moduli of torsion-free \gtstr s.

\subsection{The moduli space of torsion-free $G_{2}$-structures}

Let $M$ be a compact \gtmfd, $\calx$~the space of torsion-free \gtstr s on $M$
and $\cald$ the group of diffeomorphisms of $M$ isotopic to the identity.
The group $\cald$ acts on $\calx$, and the quotient $\defstr=\calx/\cald$
is the \emph{moduli space of torsion-free \gtstr s}.
Since torsion-free \gtstr s are represented by closed forms there is a
well-defined projection $\defstr \to H^{3}(M,\RE)$ via the de Rham cohomology.

One way to extend the definition of~$\defstr$ to an EAC \gtmfd\ $M$, with
\gtstr~$\check\varphi$ say, is to set $\calx$ to be the space of EAC
torsion-free \gtstr s 
on $M$ exponentially asymptotic to $\check\varphi$ along the cylindrical end.
The group $\cald$ is now taken to be the group of diffeomorphisms of~$M$
isotopic to the identity and on the cylindrical end exponentially
asymptotic to the identity map. Then $\defstr=\calx/\cald$ is the 
\emph{moduli space of torsion-free \gtstr s asymptotic to a fixed
cylindrical \gtstr}. It can be shown that for every $\varphi$
exponentially asymptotic to $\check\varphi$ the de Rham cohomology class
$[\varphi-\check\varphi]$ can be represented by a compactly supported closed
3-form on~$M$. (More generally, one can define a moduli space for
\gtstr s on~$M$ whose asymptotic model is allowed to vary, see~\cite{jn1}
for the details.)

\begin{thm}
\label{maincptthm}
{\rm (i)}
Let $M$ be a compact 7-manifold admitting torsion-free \gtstr s. Then the
moduli space $\defstr$ of torsion-free \gtstr s on $M$ is a smooth manifold,
and the map
$$
\pi: \varphi\cald\in\defstr \to [\varphi]\in H^{3}(M,\RE)
$$
is a local diffeomorphism.

{\rm (ii)}
Let $(M,\check\varphi)$ be an EAC \gtmfd. Then the moduli space $\defstr$ of
torsion-free \gtstr s on $M$ asymptotic to $\check\varphi$ is a smooth
manifold, and the map to affine subspace
$$
\pi: \varphi\cald\in\defstr \to [\varphi]\in [\check\varphi]+H^{3}_0(M,\RE)
\subset H^3(M,\RE)
$$
is a local diffeomorphism. Here $H^{3}_0(M,\RE)\subset H^{3}(M,\RE)$ denotes
the subspace of cohomology classes represented by compactly supported closed
3-forms. 
\end{thm}
The clause (i) is proved in \cite[Theorem $10.4.4$]{joyce00} and (ii)
in~\cite[Theorem 3.2 and Corollary 3.7]{jn1}.

The torsion-free \gtstr s discussed in this paper are obtained as a
perturbation of some closed stable 3-forms $\tv_s$ by adding a
`small' exact form. In particular, a \gtstr\ induced by $\tv_s$
necessarily has small torsion. Our next result shows that two closed stable
3-forms, which are in the same de Rham cohomology class and have small
torsion, will define the same point in $\defstr$ whenever their difference
is also small.

\begin{prop}
\label{unresprop}
Suppose that a 7-manifold~$M$ is either compact or has a cylindrical end.
For $i = 0,1$ let $\tv_{i}$ be a closed stable 3-form defining a \gtstr{}
and a metric $\tilde{g}_i=g(\tv_i)$ and Hodge star $*_i$ on~$M$.
If $M$ has a cylindrical end, suppose further that $\tv_i$ are EAC \gtstr s
and that $\tv_0-\tv_1$ decays to zero with all derivatives along the end.

Let $\psi_{i}$ be smooth $3$-forms such that
$d{*_i}\psi_i=d{*_i}\tv_i$ and
suppose that each $(\tv_{i},\psi_i)$ satisfies the
hypotheses \ref{psiestitem}--\ref{curvestitem} in Theorem~\ref{eacperturbthm},
relative to the metric~$\tilde{g}_i$.
Let $\varphi_{i}$ be the torsion-free \gtstr s
defined by Theorem~\ref{eacperturbthm} using $(\tv_{i}, \psi_{i})$.

Finally suppose that the 3-form $\tv_{0}-\tv_{1}$ is exact and
\[ \norm{\tv_{0}-\tv_{1}}_{L^{2}} < \lambda \tpar^{4}, \quad
\norm{\tv_{0}-\tv_{1}}_{C^{0}} < \lambda \tpar^{1/2}, \quad
\norm{\tv_{0}-\tv_{1}}_{\sobf{1}} < \lambda,  \]
where the norms are defined using the metric $\tilde{g}_0$.

Then for each sufficiently small $\tpar>0$, the torsion-free \gtstr s
$\varphi_i$ are isomorphic and define the same point in~$\defstr$.
\end{prop}
Recall from Remark~\ref{cptperturbrmk} that in the case when $M$ is compact
the statement of Theorem~\ref{eacperturbthm}
recovers~\cite[Theorem~11.6.1]{joyce00}.

\begin{proof}
Let $\tv_{1}-\tv_{0}=d\eta$, $\eta\in\Omega^2(M)$ and set
$\tv_{u}=\tv_0+u\,d\eta$, for $u \in [0,1]$. If $0<s<s_0$ for a sufficiently
small $s_0>0$ independent of the choice of $\tv_j$ then $\tv_{u}$ induces a
well-defined path of \gtstr s on~$M$. Define a path of 3-forms
$$
\psi'_u = \tv_{u} + 
*_u\bigl((1-u)*_{0}\!(\psi_{0} - \tv_{0})+
u*_{1}\!(\psi_{1} - \tv_{1})\bigr),
$$
where $*_{u}$ is the Hodge star of the metric defined by~$\tv_{u}$.
Then $\psi'_0=\phi_0$ and $\psi'_1=\psi_1$ and
$d{*_u}\psi_u=d{*_u}\varphi_u$, for each $u \in [0,1]$.

By our hypothesis, $(\tv_{u}, \psi'_{u})$ satisfy for $u=0$ and $u=1$, all
the estimates required in Theorem~\ref{eacperturbthm}. The left-hand
sides of these estimates depend continuously on~$u$. Therefore, by choosing
a smaller $s_0>0$ if necessary we obtain that the estimates on
$(\tv_u,\phi'_u)$ are satisfied for every $u\in [0,1]$ and
Theorem~\ref{eacperturbthm} produces a path of torsion-free \gtstr s
$\varphi_u$, connecting the given $\varphi_i$, $i=0,1$. A standard argument
verifies that $\varphi_u$ is continuous in~$u$.

By the construction, the de Rham cohomology class of $\tv_u$ is independent
of $u\in [0,1]$. By Theorem \ref{maincptthm}, the path in the moduli space
$\defstr$ defined by $\varphi_u$ must be locally constant. It follows that 
$\varphi_0$ and $\varphi_1$ define the same point in $\defstr$
and the respective \gtstr s are isomorphic.
\end{proof}

\subsection{Deformations and gluing. Proof of Theorems \ref{pullthm}
and~\ref{pullcor}}
We require one more ingredient for proving Theorem~\ref{pullthm}.
The second author \cite{jn3} shows that any small torsion-free deformation
of $\Phi(\varphi_{+}, \varphi_{-}, L)$ is, up to an isomorphism, obtainable
by gluing some small deformations of~$\varphi_{\gi}$. More important to the
present discussion is the following local description from the proof of
that result.

There are pre-moduli spaces $\gical{R}$ of EAC torsion-free \gtstr s near
$\varphi_{\gi}$, i.e. a submanifold of the space of EAC \gtstr s which is
homeomorphic to a neighbourhood of $\varphi_{\gi}$ in the moduli space
of EAC \gtstr s on $M_{\gi}$. The subspace
$\ycal{R} \subseteq \onecal{R} \times \twocal{R}$ of matching pairs is
a submanifold. The connected-sum construction gives a well-defined map
$\Phi$ from $\ycal{R} \times (L_{1}, \infty)$ (for $L_{1} > 0$ sufficiently large)
to the moduli space $\defstr$ of torsion-free \gtstr s on $M$. It is best
studied in terms of the composition with the local diffeomorphism
$\defstr \to H^{3}(M)$,
\begin{equation*}
\Phi_H : \ycal{R} \times (L_{1}, \infty) \to H^{3}(M) .
\end{equation*}

Topologically $M = M_{+} \cup M_{-}$. Consider the Mayer-Vietoris sequence
\begin{equation}
\label{mayereq}
\cdots \longrightarrow
H^{m-1}(X) \stackrel{\delta}{\longrightarrow}
H^{m}(M) \stackrel{i_{+}^{*} \oplus i_{-}^{*}}{\longrightarrow}
H^{m}(M_{+}) \oplus H^{m}(M_{-})
\stackrel{j_{+}^{*} - j_{-}^{*}}{\longrightarrow} H^{m}(X) \longrightarrow
\cdots ,
\end{equation}
where $j_{\gi} : X \to M_{\gi}$ is the inclusion of the cross-section and
$i_{\gi} : M_{\gi} \to M$ is the inclusion in the union (these maps are
naturally defined up to isotopy).

The cohomology class of the glued \gtstr{} satisfies
$i^{*}_{\gi}\Phi_H(\varphi_{+}, \varphi_{-}, L) = [\varphi_{\gi}]$.
Also $\contral \Phi_H(\varphi_{+}, \varphi_{-}, L) = 2\delta([\omega])$, where
$\omega$ denotes the Kähler form of the Calabi--Yau structure on $X$ defined
by the common asymptotic limit of $\varphi_{\gi}$.
Thus, if we let $\ycal{R}'$ be the submanifold
\[ \ycal{R}' = \{ (\psi_{+}, \psi_{-}) \in \ycal{R} :
i^{*}_{\gi}\psi_{\gi} = i^{*}_{\gi}\varphi_{\gi} \} , \]
then the restriction of $\Phi_H$ to $\ycal{R}' \times (L_{1}, \infty)$ takes
values in the affine subspace $K = [\varphi] + \delta(H^{2}(X))$, 
and can be written as
\begin{equation}
\label{coneeq}
\Phi_H : \ycal{R}' \times (L_{1}, \infty) \to K, \;\:
(\varphi'_{+}, \varphi'_{-}, L)
\mapsto F(\varphi'_{+}, \varphi'_{-}) + 2L\delta([\omega']) ,
\end{equation}
where $\omega'$ is the Kähler form of the common boundary value of
$(\varphi'_{+},\varphi'_{-}) \in \ycal{R}$ and $F : \ycal{R}' \to K$ is smooth.
It is explained in \cite[\S 5]{jn3} that the image of
$\ycal{R}' \to \delta(H^{2}(X)),
(\varphi'_{+}, \varphi'_{-}) \mapsto \delta([\omega'])$
is a submanifold transverse to the radial direction, so that
\eqref{coneeq} is diffeomorphism onto its image, which contains
an open affine cone in $K$ (if $L_{1}$ is large enough).

\begin{proof}[Proof of Theorem \ref{pullthm}]
Recall that the torsion-free \gtstr s $\phi_s(L)$ are obtained by
perturbing the closed \gtstr s $\tv_s(L)$ with small torsion, which are
in turn defined by stretching the cylindrical neck $X \times I$ of $\tv_s$
by a length $2L$. Their cohomology classes are
$[\phi_s(L)] = [\tv_s(L)] = [\varphi_s] + 2L\delta([\omega])$, where
$\omega$ is the Kähler form on $X$, so the image of the path $\phi_s(L)$
in $H^{3}(M)$ is an affine line with slope $2\delta([\omega])$.

We also defined torsion-free EAC \gtstr s $\varphi_{s,\gi}$ on $M_{\gi}$
by perturbing the \gtstr s $\tv_{s,\gi}$ obtained from $\tv_s$ via
decomposition~\eqref{decomp} of~$M$.
The gluing Theorem~\ref{gluing} applied to $\varphi_{s,+}$ and $\varphi_{s,-}$
defines a path $\Phi(\varphi_{s,+}, \varphi_{s,-}, L)$ of torsion-free \gtstr
s on $M$. The restrictions satisfy
$i_{\gi}^{*}[\Phi(\varphi_{s,+}, \varphi_{s,-}, L)] = i_{\gi}^{*}[\varphi_s]$,
so the image of the path in $H^{3}(M)$ lies in
the affine space $K = [\varphi_s] + \delta(H^{2}(X))$.
This is an affine line with the same slope $2\delta([\omega])$.

Our aim is to show that for every large $L$ there is a small deformation
$(\varphi'_{s,+}(L), \varphi'_{s,-}(L))$ of $(\varphi_{s,+}, \varphi_{s,-})$
and $L+\eps_L$ 
at a bounded distance from $L$, so that $\phi_s(L)$ is isomorphic
to \mbox{$\Phi(\varphi'_{s,+}(L), \varphi'_{s,-}(L), L+\eps_L)$.} We prove
this by appealing to
to Proposition \ref{unresprop}, showing first that we can find a small
deformation such that the glued $\Phi(\varphi'_{s,+}(L), \varphi'_{s,-}(L),
L+\eps_L)$
has the same cohomology class as $\phi_s(L)$, and then checking that
the gluing is close to $\tv_s(L)$ in the relevant norms.

The difference between the cohomology classes $[\phi_s(L)]$ and
$[\Phi(\varphi_{s,+}, \varphi_{s,-}, L)]$ is independent of~$L$.
Therefore, for each sufficiently large $L$, there is
an $L+\eps_L$ of bounded distance to~$L$ and a matching pair
$(\varphi'_{s,+}(L), \varphi'_{s,-}(L)) \in \ycal{R}'$,
such that $\phi_s(L)$ is cohomologous to
the glued \gtstr\ $\Phi(\varphi'_{s,+}(L), \varphi'_{s,-}(L), L+\eps_L)$.
In fact, because the RHS of (\ref{coneeq}) is dominated by the
$2L\delta([\omega])$ term for large $L$, the distance between
$(\varphi'_{s,+}(L), \varphi'_{s,-}(L))$ and $(\varphi_{s,+}, \varphi_{s,-})$ is of
order $1/L$, as $L\to\infty$, measured in the $C^{1}$ norm (since $\ycal{R}$
has finite dimension all sensible norms are Lipschitz equivalent).
Hence the difference between $\Phi(\varphi_{s,+}, \varphi_{s,-}, L)$ and
$\Phi(\varphi'_{s,+}(L), \varphi'_{s,-}(L), L)$ is of order $1/L$ in $C^{0}$
norm. As the volume growth is of order $L$ it follows also that the difference
is of order $L^{-1/2}$ in $L^{2}$-norm, and order $L^{-13/14}$ in
$L^{14}_{1}$-norm.

Now $\phi_s(L)$ and $\Phi(\varphi'_{s,+}(L), \varphi'_{s,-}(L),
L+\eps_L(L))$
are both torsion-free perturbations of $\tv_s(L)$ within its cohomology class,
so we can try and use Proposition
\ref{unresprop} to show that they are diffeomorphic.
For large $L$ the difference between
$\Phi(\varphi'_{s,+}(L), \varphi'_{s,-}(L), L+\eps_L)$
and $\tv_s(L)$ is dominated by the difference between $\tv_{s,\gi}$
and $\varphi_{s,\gi}$, which is estimated in terms of $\tpar$
in (\ref{eacdetasmalleq}). Therefore if $s$ is sufficiently small
then for all sufficiently large $L$ the estimates required to
apply Proposition \ref{unresprop} are satisfied, and
\[ \Phi(\varphi'_{s,+}(L), \varphi'_{s,-}(L), L+\eps_L) \cong \phi_s(L) . \]
This completes the proof of Theorem~\ref{pullthm}.
\end{proof}

\begin{proof}[Proof of Theorem~\ref{pullcor}]
We know from the argument of Theorem~\ref{pullthm} and the preceding remarks
that the pair $\varphi'_{s,+}(L),\varphi'_{s,-}(L)$, for each
$L>L_1$, is contained in the pre-moduli space $\ycal{R}'$ which we may
assume connected. As discussed earlier in this subsection, the map
$\Phi(\varphi_+, \varphi_-, L)$ induces a continuous function from
$\ycal{R}'\times(L_1,\infty)$ to the $G_2$ moduli space for~$M$.
We find that, for $L>L_1$, the torsion-free \gtstr{}
$\Phi(\varphi_{s,+},\varphi_{s,-},L)$  is a deformation of
$\Phi(\varphi'_{s,+}(L),\varphi'_{s,-}(L),L)$.

By Theorem~\ref{pullthm}, we may further replace
$\Phi(\varphi'_{s,+}(L), \varphi'_{s,-}(L), L)$, with the torsion-free
\gtstr{} $\phi_s(L-\eps_L)$, assuming sufficiently large~$L$. We saw
above that the cohomology class $[\phi_s(L)]$ depends continuously
on~$L$ and it is not difficult to check, using Theorem~\ref{maincptthm}(i)
that the forms $\phi_s(L)$ define a continuous path in the $G_2$
moduli space for~$M$. Thus we may further replace $\phi_s(L-\eps_L)$
by the torsion-free \gtstr{} $\phi_s(0)$. We claim that the latter is
isomorphic to the \gtstr~$\varphi_s$.

By definition just before Theorem~\ref{pullthm}, $\phi_s(0)$ is a
perturbation of $G_2$ 3-form $\tv_s(0) = \tv_s$ given by
Proposition~\ref{intermed.phi} and
$\phi_s(0)-\tv_s$ is exact. On the other hand, recall from
\eqref{joyce.g2} that $\varphi_s$ is a perturbation of $G_2$ 3-form 
$\phin_s$ by an exact form. The latter two exact forms may be assumed
`small' in the sense of~\eqref{tor} by choosing a small~$s$. Furthermore,
$\tv_s-\phin_s$ is exact by Proposition \ref{intermed.phi}\ref{exactitem} and
the argument of \S\ref{four.three} (\eqref{tor} and \eqref{drcsmalleq}) again
shows that $\tv_s-\phin_s$ is small. Proposition~\ref{unresprop} now ensures
that $\varphi_s$ and $\phi_s(0)$ are diffeomorphic, for every
sufficiently small~$s$.
\end{proof}

\section{Connected sums of EAC $G_2$-manifolds}
\label{connectsums}

We now revisit the orbifold $T^7/\Ga$ discussed in~\S\ref{simplesub} but this
time we shall split $T^7/\Ga$ into two connected components,
$\hat{M}_{0,\pm}$ say, along a different orbifold hypersurface $\hat{X}_0$ which
is the image of the 6-torus $\hat{T}^6=\{x_5\equiv 1/8 \mod\ZE\}\subset T^7$.
(As before, $x_k$ modulo $\ZE$ denote the standard coordinates on $T^7$
induced from~$\RE^7$.) As remarked in \S\ref{more.examples}, this choice
does not produce an irreducible EAC \gtmfd\ but is interesting for its
relation to the compact \gtmfd s and EAC Calabi--Yau 3-folds constructed
in~\cite{kovalev03,kovalev-lee08}.

More precisely, we shall show that the corresponding EAC \gtmfd s
$\hat{M}_\gi$ are of the form $S^1 \times W$, where $W$ is a known complex
3-fold obtained by the algebraic methods of \cite{kovalev-lee08} with an EAC
Calabi--Yau structure coming from a result in~\cite{kovalev03}.
Application of Theorem~\ref{pullcor} then shows that the \gtstr\ on $M$
constructed in~\cite{joyce00} by resolution of singularities of~$T^7/\Ga$
is a deformation of the \gtstr\ obtainable from~\cite{kovalev03} by regarding
$M$ as generalized connected sum of EAC \gtmfd s $\hat{M}_\gi$.

\subsection{A $G_2$-manifold with holonomy $SU(3)$.}
\label{with.K3}

Recall that the singular locus of $T^7/\Ga$ consists of 12 disjoint copies
of $T^3$, the union of 3 subsets of 4 copies of $T^3$ corresponding
to the fixed point set of, respectively, the involutions $\alpha,\beta,\gamma$
defined in~\eqref{simplegammaeq}. Each of the 4 copies of $T^3$ in
the singular locus of $T^7/\Ga$, arising from the fixed points of~$\beta$,
intersects $\hat{X}_0$ in a 2-torus. The other 8 copies of $T^3$ in the
singular locus do not meet $\hat{X}_0$. Let $\hat{M}_{0,+}$ denote the
connected component of $(T^7/\Ga)\setminus \hat{X}_0$ containing the image of
$\{x_5=0\}$. Then $\hat{M}_{0,+}$ contains all the 3-tori coming from the
fixed point set of~$\alpha$, whereas those coming from~$\gamma$ are in the image
of $\{x_5=\frac14\}$ and contained in~$\hat{M}_{0,-}$.

It is easy to see that
$\hat{X}_0=\hat{T}^6/\gen{\beta}\cong (T^4/{\pm 1})\times T^2$ and
that the orbifolds $\hat{M}_{0,\pm}$ are diffeomorphic, via the
involution of $T^7/\Ga$ induced by the map
\begin{equation}\label{shift.5}
(x_1,x_2,x_3,x_4,,x_5,x_6,x_7)\mapsto
(x_1,x_2,x_3,x_4,x_5+\quart,x_6,x_7).
\end{equation}
The above is quite similar to the discussion in \S\ref{simplesub}
and~\S\ref{EACex}. In particular, it can be shown that the
map~\eqref{shift.5} induces an isometry of the EAC $G_2$-manifolds
$\hat{M}_{\pm}$ constructed from $\hat{M}_{0,\pm}$ (compare
Remark~\ref{isometry}).

Notice also that the pre-image of $\hat{M}_{0,+}$ in
$T^7/\gen{\alpha,\beta}$ consists of two connected components and $\gamma$
maps one of these diffeomorphically onto the other. In light of this, we
can identify
$\hat{M}_{0,+}\cong
\bigl(\{|x_5|<\textstyle{\frac18}\}\times T^6\bigr)/\gen{\alpha,\beta}$,
and disregard $\gamma$ when restricting attention to $\hat{M}_{0,+}$. Replacing 
the interval $[-\frac18,\frac18]$ by a copy of $\RE$, with the coordinate
still denoted by $x_5$, is equivalent to attaching a cylindrical end
to~$\hat{M}_{0,+}$. We have a diffeomorphism
\begin{equation}\label{no.gamma}
\hat{M}_{0,+}\cong
\bigl((\RE_{x_5}\times T^5\bigr)/\gen{\alpha,\beta}\bigr)\times S^1_{x_1}.
\end{equation}
We see at once that the resolution of singularities of $\hat{M}_{0,+}$ amounts
to resolving a 6-dimensional orbifold. We shall relate the latter resolution
to blowing up {\em complex orbifolds}. Identify $\RE^7\cong\RE\times\CX^3$
using a real coordinate and three complex coordinates,
\begin{equation}\label{cx.coord}
\theta=x_1,\quad z_1=x_5+ix_4,\quad z_2=x_2+ix_3,\quad z_3=x_6+ix_7.
\end{equation}
In these coordinates, the involutions $\alpha,\beta$ are
{\em holomorphic in $z_k$}
$$
\alpha(\theta,z_1,z_2,z_3)=(\theta,-z_1,z_2,-z_3),
\qquad
\beta(\theta,z_1,z_2,z_3)=(\theta,z_1,-z_2,\half-z_3).
$$

For the first step of the procedure explained in~\S\ref{EACex} we consider
$\RE_{x_5}\times\hat{T}^6/\gen{\beta}$. It is well-known that the resolution
of singularities of $T^4/{\pm 1}$ using Eguchi--Hanson spaces (see
p.\pageref{EH}) produces a Kummer K3 surface, $Y$ say. The Kummer
construction defines on~$Y$ a one-parameter family of torsion-free
$SU(2)$-structures, i.e.\ Ricci-flat Kähler structures, with a limit
corresponding to the flat hyper-Kähler structure on~$T^4/{\pm 1}$ induced
from the Euclidean~$\RE^4$ \cite{LS}.
Cf.~\eqref{EHmetr}; the parameter, still denoted by $s > 0$, is
proportional to the diameter of the exceptional divisors on~$Y$.
We thus obtain $S^1_{x_1}\times S^1_{x_4}\times\RE_{x_5}\times Y$ with
a product torsion-free \gtstr\ induced by a Kummer hyper-Kähler structure
on~$Y$ (cf.~\eqref{g2su2}).

The Kummer construction can be performed $\alpha$-equivariantly, so that 
$\alpha$ induces an involution on~$Y$, say $\rho_\alpha$, which
preserves the $SU(2)$-structure. The quotient~\eqref{no.gamma}
takes the form $S^1_\theta\times Z_0$, where
$Z_0=(\RE_{x_5}\times S^1_{x_4}\times Y)/\gen{\alpha}$ is a well-defined
complex orbifold. Noting that $\RE_{x_5}\times S^1_{x_4}$ is biholomorphic to
$\CX^{\times}=\CX\setminus\{0\}$, we can extend $\alpha$ holomorphically to an
involution of $Y\times\CP^1$ (identifying $\CP^1\cong\CX\cup\{\infty\}$). The
restriction of $\alpha$ to $\CP^1$ may be written as $\zeta\mapsto 1/\zeta$,
where $\zeta=\exp(2\pi i z_1)$; it maps $0$ and~$\infty$ to each other
and fixes precisely two points $\pm 1$, both in the image of
$\RE_{x_5}\times S^1_{x_4}$ (the circle $S^1_\theta$ does not yet concern
us). We can write $Z_0$ and its compactification~$Z$ as
\begin{equation}\label{orbi.Z}
Z_0=(Y\times\CX^\times)/\gen{\alpha},\qquad Z=(Y\times\CP^1)/\gen{\alpha}
\end{equation}
and it is not difficult to check that $Z_0$ is the complement in $Z$ of an
{\em anticanonical divisor} $D$ biholomorphic to the K3 surface~$Y$.
The quotient of $\CP^1$ by the involution $\alpha|_{\CP^1}$ is
biholomorphic to $\CP^1$ and we shall still denote the images of the fixed
points by $\pm 1$. It follows that the second projection on $Y\times\CP^1$
descends to a holomorphic map 
\begin{equation}\label{z.fibr}
p:Z_0\to\CP^1
\end{equation}
with fibres biholomorphic to~$Y$, except that the two fibres over $\pm 1$
are biholomorphic to the quotients $Y/\gen{\rho_\alpha}$.

Denote by $\kappa_I$, $\kappa_J$, $\kappa_K$ a triple of closed 2-forms
encoding the $\rho_\alpha$-invariant $SU(2)$-structure on~$Y$.
Here $\kappa_I$ is the Kähler form of the Ricci-flat Kähler metric and
$\kappa_J+i\kappa_K$ is a nowhere-vanishing holomorphic (2,0)-form,
sometimes called a `holomorphic symplectic form', which is unique up to a
constant complex factor. We shall always require
$\kappa_I^2=\kappa_J^2=\kappa_K^2$.

Observe that necessarily
$\rho_\alpha^*(\kappa_J+i\kappa_K)=-\kappa_J-i\kappa_K$, so $\rho_\alpha$
acts by $-1$ on $H^{2,0}(Y)$. The latter makes $Y$ into a K3
surface with `non-symplectic involution' in the sense of \cite{AlNi}, see
also~\cite{Ni}. A general property of this class of K3 surfaces is that
the sublattice of $H^2(Y,\ZE)$ fixed by $\rho_\alpha^*$ has signature
$(1,t_-)$, so we must have $\rho_\alpha^*(\kappa_I)=\kappa_I$, because
a Ricci-flat Kähler metric on~$Y$ is uniquely determined by the
cohomology class of its Kähler form.

In order to compute some topological invariants later, we shall need some
algebraic invariants of non-symplectic involutions, taken from \cite{AlNi}.
One invariant is defined as the rank $r$ of the sublattice $L_\rho$ of the
Picard lattice of $Y$ fixed by $\rho_\alpha$. It can be shown $L_\rho$ has
a natural embedding into its dual lattice $L_\rho^*$ and the quotient has
the form $L_\rho^*/L_\rho\cong(\ZE_2)^a$.  The integer $a$ is another
invariant that we shall need.

We determine the values of $r,a$ in the present example from the classification
of K3 surfaces with non-symplectic involution in~\cite{Ni}, which
includes a description of the fixed point set of~$\rho$. Since the fixed point
set of $\alpha$ has 4 components and the induced involution on $\CP^1$ fixes 2
points, we must have that the $\rho_\alpha$ fixes precisely two disjoint
complex curves and each of these has genus~1. In this situation, there is only
one possibility $r=10$, $a=8$ allowed by the classification of fixed point
sets, \cite[\S 4]{Ni} or \cite[\S 6.3]{AlNi}.

A neighbourhood of each singular point in~$Z_0$ is diffeomorphic to
$(\CX^2/{\pm 1})\times\CX$ and the singularities of $S^1_\theta\times Z_0$ may
be resolved in an $S^1_\theta$-invariant way by gluing in an Eguchi--Hanson
space, similarly to several instances discussed in \S\ref{exsec}
and~\S\ref{furthersec}. The two-step procedure of \S\ref{exsec} now
produces a 7-manifold $\hat{M}_+=S^1\times W$ with an $S^1$-invariant,
product \gtstr\ having `small' torsion. The torsion-free \gtstr\ on
$\hat{M}_+$ obtained by Theorem~\ref{eacperturbthm} is necessarily of
product type \eqref{g2su3} induced by a EAC Calabi--Yau structure on~$W$.

We shall now show that after slightly changing some details of the method
of~\S\ref{exsec} the same torsion-free \gtstr\ on $\hat{M}_+$ can be
recovered, up to an isomorphism, by constructing an EAC Calabi--Yau
structure on~$W$ using the method of~\cite{kovalev03,kovalev-lee08}. Recall
from \S\ref{su2sub} that a Calabi--Yau structure on a 6-manifold may be
determined by the complex structure (or, equivalently, the real part of a
non-vanishing holomorphic 3-form) and the Kähler form.

The manifold $W$ has a `natural' complex structure defined by blowing-up the
singular locus of the complex orbifold~$Z_0$. This is an instance of a
general construction of quasiprojective complex $3$-folds with trivial
canonical bundle from $K3$ surfaces with non-symplectic involution.

\begin{prop}[{\cite[\S 4]{kovalev-lee08}}]\label{W}
Suppose that $\rho$ is a non-symplectic involution of a K3 surface $Y$
with invariants $r,a$ and with a non-empty set of fixed points. Suppose that
$\tau$ is a holomorphic involution of $\CP^1$ fixing precisely two
points. Let $\oW$ be the blow-up of the singular locus of
$(Y\times\CP^1)/(\rho,\tau)$ and let $D\subset\oW$ be the pre-image of
$Y\times\{p\}$, for some $p\in\CP^1$ with $\tau(p)\neq p$.

Then both $\oW$ and $W=\oW{\setminus}D$ are non-singular and
simply-connected and $D$ is an anti\-canonical divisor (biholomorphic to~$Y$)
in $\oW$ with the normal bundle of $D$ holomorphically trivial.
Also, $b^2(\oW)=3+2r-a$ and $b^3(\oW)=44-2r-2a$ and the pull-back map
\mbox{$\iota:H^2(\oW,\RE)\to H^2(D,\RE)$} induced by the embedding has rank~$r$.
\end{prop}

In particular, $W$ admits nowhere-vanishing holomorphic $(3,0)$-forms.
An example of such form is obtained by starting on
$\RE_{x_5}\times S^1_{x_4}\times Y$ with the wedge product of
$d\zeta/\zeta=dz_1=dx_5+idx_4$ and the `obvious' pull-back of a holomorphic
symplectic form on~$Y$. This $(3,0)$-form is $\alpha$-invariant and
descends to $Z_0$. Denote its pull-back via the blow-up $W\to Z_0$ by
$\Omega'+i\Omega''$. This form is well-defined and may be alternatively
obtained using the following resolution of singularities commutative
diagram
$$
\xymatrix{
\widetilde{W} \ar[d] \ar[r] &W \ar[d] \\
Y\times\CX^\times \ar[r]     &Z_0}
$$
where $\widetilde{W}\to Y\times\CX^\times$ is the blow-up of the fixed
point set of~$\alpha$ and $\widetilde{W}\to W$ is the quotient map for the
involution of $\widetilde{W}$ induced by~$\alpha$.

We next construct a suitable Kähler form on~$W$. The form $idz_1\we d\bar{z}_1
+\kappa_I$ defines an \mbox{$\alpha$-invariant} Ricci-flat Kähler metric
on $Y\times\CX^\times$. Pulling back to $W$ similarly to above, we obtain a
2-form $\omega_0$ which is a well-defined Kähler form away from the
exceptional divisor $E$ on~$W$. The exceptional divisors on~$Y$ arising from
the Kummer construction induce divisors on~$Z_0$, by taking a product with
$\CX^\times$ and dividing out by~$\alpha$. The proper transform of these
defines a divisor, F say, on W. By choosing the parameter $s$ in the Kummer
construction sufficiently small we achieve that the curvature of $\omega_0$
is small away from a tubular neighbourhood of~$F$.
Note that $F$ does not meet $E$ because the fixed point sets of $\alpha$ and
$\beta$ do not meet (see~\S\ref{simplesub}). We can choose disjoint tubular
neighbourhoods of $E$ and of $F$. Then on the intersection of a tubular
neighbourhood $V$ of $E$ with the domain of $\omega_0$ the metric $\omega_0$ is
close to flat whenever $s$ is sufficiently small.

On the other hand, by taking a product of the Eguchi-Hanson
metric~\eqref{EHmetr} (with the same value of $s$) and the standard Kähler
metric on an open domain in~$\CX$ we obtain a Kähler form $\omega_{EH}$
which is defined near $E$. With an appropriate
choice of~$V$, we can smoothly interpolate between the Kähler potentials of
$\omega_0$ and $\omega_{EH}$ to obtain a closed real $(1,1)$-form $\omega_s$,
so that $\omega_s^3\neq 0$ and $\omega_s$ is a well-defined Kähler form on~$W$.
An argument similar to that in \S\ref{four.three} shows we can perform this 
construction of $\omega_s$ without introducing any more torsion of the
corresponding \gtstr{} than we would if $\omega_0$ was actually flat.
That is, the closed $S^1$-invariant \gtstr{} $\varphi'_{W,s}=\Omega' +
d\theta \wedge \omega_s$ on the 7-manifold $\hat{M}_+$ has `small' torsion
in the sense of Proposition~\ref{intermed.phi}. Here we take $\psi=\psi_s=
\Theta(\varphi'_{W,s}) - d\theta\we\hat\Omega''-\half\omega_s\we\omega_s$.
Then Theorem~\ref{eacperturbthm} produces an $S^1$-invariant
torsion-free EAC \gtstr\ $\varphi_{W,s}+d\eta_s$ on $S^1\times W$
determined by an EAC Calabi--Yau structure on~$W$ (cf.~\eqref{g2su2}).
Remark that the starting \gtstr\ with small torsion and the choice of
$\psi$ may differ by a `small amount' from those described
in~\S\ref{exsec}, but the resulting torsion-free \gtstr s are isomorphic
by Proposition~\ref{unresprop}.

The latter EAC Calabi--Yau structure is asymptotic on the end
$\RE_{>0}\times S^1\times Y$ of~$W$ to the product Calabi--Yau structure
corresponding to the 
hyper-Kähler structure on~$Y$ and
is obtained by the following `non-compact version of the Calabi conjecture'.

\begin{thm}[{\cite[\S 3]{kovalev03}}]\label{bCY}
Let $(\oW,\overline{\omega})$ be a simply-connected complex 3-fold and
suppose that a K3 surface $D\subset\oW$ is an anticanonical divisor with
the normal bundle of $D$ holomorphically trivial and
$W=\oW\setminus D$ simply-connected.
Let $\kappa_I$, $\kappa_J$, $\kappa_K$ be a triple of closed 2-forms
inducing a Calabi--Yau structure on~$D$, as above.

Suppose that $\tilde\omega$ is a Kähler form on~$W$ which is
asymptotically cylindrical in the following sense. There is a meromorphic
function $z$ on $\oW$ vanishing to order one precisely on~$D$. On the
region $\{0<|z|<\eps\}$, for some $\eps>0$, $\omega$ has the asymptotic form
$$
\kappa_I+dt\we d\theta+d\tilde\psi
$$
where $\exp(-t-i\theta)=z$
and a 1-form $\tilde\psi$ is exponentially decaying with all derivatives
as $t\to\infty$.

Then $W$ admits a asymptotically cylindrical Ricci-flat Kähler metric with
Kähler form $\omega$ and a nowhere-vanishing holomorphic $(3,0)$-form
$\Omega'+i\Omega''$ such that
$$
\omega=\tilde\omega+i\p\bar\p\psi_\infty
$$
and $\Omega$ on the region $\{0<|z|<\eps\}$ has the asymptotic form
$$
(\kappa_J+i\kappa_K)\we (dt+id\theta)+d\Psi_\infty,
$$
where $\psi_\infty,\Psi_\infty$ are exponentially decaying with all derivatives
as $t\to\infty$.
\end{thm}

In the present example, we have $\tilde\psi=0$ by construction.
The function $\psi_\infty$ is unique by~\cite[Propn.~3.11]{kovalev03}.
The uniqueness of $d\Psi_\infty$ follows from the uniqueness, up to a
constant factor, of a non-vanishing holomorphic 3-form on $W$ with a simple
pole along $D=\oW{\setminus}W$.
Thus the \gtstr\ obtained by application of Theorem~\ref{eacperturbthm} to
the cylindrical end manifold $\hat{M}_+=S^1\times W$ with \gtstr{}
$\varphi'_{W,s}$ is unique and may be recovered from a blow-up of complex
orbifold and the Calabi-Yau analysis.

The Betti numbers for our example of $\hat{M}_+$ may be determined
from those of $\oW$ using Proposition~\ref{W} as we know that
$r=10$, $a=8$. We obtain
$$
b^{3}(\oW) = 44-20-16=8\text{ and }
b^2(\oW)=3+20-8=15,
$$
and then, using the Mayer--Vietoris exact sequence for $\oW=W\cup D$ similarly
to~\cite[\S 8]{kovalev03} and \cite[\S 2]{kovalev-lee08},
$$
b^2(W)=b^2(\oW)-1=14\text{ and }
b^3(W)=b^3(\oW)+22-b^2(W)+\dim\Ker\iota=20,
$$
using also the rank-nullity for~$\iota$. Therefore,
$$
b^2(\hat{M}_+)=14\text{ and }b^3(\hat{M}_+)=34
$$
by the Künneth formula.

The Betti numbers of $W$ and $\hat{M}_+$ can also be recovered using the
method explained at the end of~\S\ref{five.one}.

\subsection{The connected sum construction of compact irreducible
  $G_2$-manifolds revisited}

Everything that we said in the previous subsection about $\hat{M}_{0,+}$
and $\hat{M}_+$ can be repeated, with a change of notation, for $\hat{M}_{0,-}$
and $\hat{M}_-$. In particular $\hat{M}_-=W\times S^1$ with a product EAC \gtstr.
However, the roles of $\alpha$ and $\gamma$ are swapped for $\hat{M}_{0,-}$
and the choice of identification $\RE^7=\RE_\theta\times\CX^3$ has to be
revised too.

For $\hat{M}_{0,-}$, we set
\begin{equation}\label{w.coord}
\theta=x_4,\quad w_1=x_5+ix_1,\quad w_2=x_2+ix_6,\quad w_3=x_7+ix_3,
\end{equation}
so that
$$
\beta(\theta,w_1,w_2,w_3)=(\theta,w_1,{\textstyle\frac{i}2}-w_2,-w_3),
\qquad
\gamma(\theta,w_1,w_2,w_3)=
(\theta,\half-w_1,w_2,\half-w_3).
$$

We are interested in the image in $\hat{X}_0$ of the 4-torus corresponding
to $x_2,x_3,x_6,x_7$. Writing
$$
\kappa^0_1=dx_2\we dx_3+dx_6\we dx_7,\quad
\kappa^0_2=dx_2\we dx_6+dx_7\we dx_3,\quad
\kappa^0_3=dx_2\we dx_7+dx_3\we dx_6,
$$
we see that with respect to the complex structure  on $\RE^4_{x_2,x_3,x_6,x_7}$
defined by $z_2,z_3$ in~\eqref{cx.coord} the Euclidean metric is Kähler
with Kähler form $\kappa^0_1$ and a $(2,0)$-form $\kappa^0_2+i\kappa^0_3$.
With respect to the complex structure of $w_2,w_3$ the Kähler
form is $\kappa^0_2$ and a $(2,0)$-form is $\kappa^0_1-i\kappa^0_3$.

It follows by the symmetry of even permutations of $x_2,x_3,x_6,x_7$ and the
equivariant properties of the Kummer construction that a similar statement
holds for a triple of 2-forms, say $\kappa_I,\kappa_J,\kappa_K$ defining
the hyper-Kähler structure on the resolution $Y$ of
$T^4_{x_2,x_3,x_6,x_7}/\gen{\beta}$.

In other words, the two Kummer K3 surfaces defined by using $z$- and
$w$-coordinates correspond to choices of two anticommuting integrable
complex structures say $I$ and $J$ coming from the hyper-Kähler structure
on~$Y$. The $\kappa_I,\kappa_J,\kappa_K$ are the Kähler forms corresponding,
respectively, to $I$,$J$,$K=IJ$.

Recall from~\S\ref{su2sub} that the product \gtstr\ on a cylinder
$\RE_t\times S^1_{\theta_+}\times S^1_{\theta_-}\times D$ corresponding to a
hyper-Kähler structure on~$D$ is induced by the 3-form
\begin{equation}\label{phi.d}
\varphi_D=d\theta_+\we d\theta_-\we dt+
d\theta_+\we\kappa_I+d\theta_-\we\kappa_J+dt\we\kappa_K.
\end{equation}
Here $\theta_+=x_1$, $\theta_-=x_4$, corresponding to
\eqref{cx.coord},\eqref{w.coord} and $x_5=t$.
The formula \eqref{phi.d} is preserved by the transformation
$$
\theta_+\mapsto\theta_-,\quad
\theta_-\mapsto\theta_+,\quad
t\mapsto -t,\qquad
\kappa_I\mapsto\kappa_J,\quad
\kappa_J\mapsto\kappa_I,\quad \kappa_K\mapsto -\kappa_K.
$$
Notice that the transformation of $\kappa$'s corresponds precisely to
changing the complex structure on~$Y$ from $I$ to~$J$ (the latter is
sometimes called a `hyper-Kähler rotation').
It follows that we have an instance of a generalized connected sum of EAC
\gtmfd s discussed in the beginning of~\S\ref{pullsec}. In fact, more is true.

We can identify, in the present case, the isomorphism between the asymptotic
models of EAC $G_2$ 3-forms on the cylindrical ends of
$\hat{M}_\pm\cong S^1_\pm\times W_\pm$. (Here $W_\pm$ are copies of $W$
defined in the previous subsection and $\pm$ refers to using, respectively,
the notation~\eqref{cx.coord} or \eqref{w.coord}.)
On the $D\cong Y$ factor the identification is an
isometry with a change of complex structure, as discussed above. The $\pm x_5$
is the parameter along cylindrical end of $\hat{M}_\pm$, respectively. Finally, the
$S^1_+$-factor with coordinate $x_1$ is identified with a circle around the K3
divisor in $\oW_-$, whereas the $S^1_-$ factor with coordinate $x_4$
corresponds to a circle around the K3 divisor in~$\oW_+$.

The matching described above between the asymptotic models of EAC \gtmfd s
$\hat{M}_\pm$ is precisely of the type studied in~\cite{kovalev03}. In
particular, the gluing Theorem~\ref{gluing} constructs an irreducible
torsion-free \gtstr\ on $M$ regarded as the generalized connected sum of the
pair $\hat{M}_\pm$ defined above, with product EAC \gtstr s induced
by the EAC Calabi--Yau structures on $W_\pm$ in the sense of Theorem~\ref{bCY}.

The `glued' $G_2$-metrics on $M$ obtainable by Theorem~\ref{gluing} are of the
type described in \cite[Theorem 5.3]{kovalev-lee08}. When $W_1$,$W_2$
are constructed from a pair of K3 surfaces with non-symplectic involution
with invariants $r_j,a_j$ and with $d_j=\dim\Ker\iota_j$ as defined in
Proposition~\ref{W}, the resulting compact \gtmfd{} $M$ has
$$
b^2(M)=d_1+d_2+\dim\bigl(\iota_1(H^2(W_+,\RE))\cap\iota_2(H^2(W_-,\RE))\bigr).
$$
Recall that we have $b^2(M)=12$ and $d_1=d_2=4$, whence the last dimension
in the right-hand side is 4. The examples explicitly discussed in
\cite{kovalev-lee08} all have the latter intersection zero-dimensional,
thus $M$ is a new example for the construction given there.

By Theorem~\ref{pullcor} and the work in~\S\ref{with.K3}, the glued
torsion-free \gtstr\ on $M$ obtainable as in~\cite{kovalev03,kovalev-lee08}
is a continuous deformation of a torsion-free \gtstr\ given by
resolving singularities of $T^7/\Ga$ according to \cite[\S 11]{joyce00}.
Therefore, the moduli space for torsion-free \gtstr s on~$M$ has a
connected component
with boundary points corresponding to two types of degenerations of
$G_2$-metrics: (1) those arising by pulling $M$ apart into a pair of EAC
\gtmfd s and (2) those developing orbifold singularities but staying
compact with volume and diameter bounded. To our knowledge, $M$ is the
first example of a compact irreducible \gtmfd\ obtainable, up to deformation,
both by the method of~\cite{joyce00} and by the method of~\cite{kovalev03}.

\end{document}